\documentclass[reqno,a4]{amsart}

\numberwithin{equation}{section}

\usepackage[latin1]{inputenc}
\usepackage[english]{babel}

\usepackage{amsmath,amsthm,amsfonts,latexsym,amssymb}
\usepackage[colorlinks]{hyperref}
\hypersetup{linkcolor=blue,citecolor=blue,filecolor=black,urlcolor=blue}
\usepackage{comment}

{ \theoremstyle{plain}
\newtheorem{theorem}{Theorem}[section]

\newtheorem{lemma}[theorem]{Lemma}
\newtheorem{corollary}[theorem]{Corollary}
  \theoremstyle{remark}
\newtheorem{remark}[theorem]{Remark}
  \theoremstyle{definition}
\newtheorem{definition}[theorem]{Definition}
}

\def\R{\mathbb{R}}
\def\Z{\mathbb{Z}}
\def\C{\mathbb{C}}

\DeclareMathOperator{\Ker}{Ker}
\DeclareMathOperator{\Rk}{Range}
\DeclareMathOperator{\spann}{span}
\def\eps{\varepsilon}
\def\ep{\varepsilon}

\def\a{\alpha}
\def\b{\beta}
\def\m{\mu}

\def\Om{\Omega}
\def\eps{\varepsilon}
\def\vphi{\varphi}
\def\vtheta{\vartheta}

\def\Hcal{\mathcal{H}}
\def\Ecal{\mathcal{E}}
\def\Qcal{\mathcal{Q}}
\def\Ucal{\mathcal{U}}

\def\icomp{\mathrm{i}}

\def\HcalC{{\Hcal_\C}}

\begin{document}

\title[]{Stable solitary waves with prescribed $L^2$-mass\\
for the cubic Schr\"odinger system\\ with trapping potentials}

\author[B. Noris]{Benedetta Noris}
\address{Benedetta Noris \newline \indent INdAM-COFUND Marie Curie Fellow,
\newline \indent Laboratoire de Math\'ematiques de Versailles, Universit\'{e} de Versailles Saint-Quentin,
\newline \indent 45 avenue des Etats-Unis,
78035 Versailles C\'edex, France.}
\email{benedettanoris@gmail.com}

\author[H. Tavares]{Hugo Tavares}
\address{Hugo Tavares \newline \indent CAMGSD, Instituto Superior T\'ecnico
\newline \indent Pavilh\~ao de Matem\'atica, Av. Rovisco Pais \newline \indent
1049-001 Lisboa, Portugal}
\email{htavares@math.ist.utl.pt}

\author[G. Verzini]{Gianmaria Verzini}
\address{Gianmaria Verzini \newline \indent Dipartimento di Matematica, Politecnico di Milano
\newline \indent piazza Leonardo da Vinci 32 \newline \indent
20133 Milano, Italy}
\email{gianmaria.verzini@polimi.it}

\date{\today}

\maketitle

\begin{abstract}
For the cubic Schr\"odinger system with trapping potentials in $\R^N$, $N\leq3$, or in bounded
domains, we
investigate the existence and the orbital stability of standing waves having components
with prescribed $L^2$-mass. We provide a variational characterization of such solutions, which
gives information on the stability through of a condition of Grillakis-Shatah-Strauss type.
As an application, we show existence of conditionally orbitally stable solitary waves when: a) the masses are small, for almost every scattering lengths, and b) in the defocusing, weakly interacting case, for any masses.
\end{abstract}

\section{Introduction}
Let $\Omega\subset\R^N$, $N \leq 3$, be either the whole space,  or a bounded Lipschitz domain, and
let us consider two trapping potentials $V_1$, $V_2$, satisfying
\begin{equation}\label{eq:V_assumptionT}
\tag{TraPot}
V_i \in \mathcal{C}(\overline{\Omega}), \quad V_i\geq0,
\quad \lim_{|x|\to\infty} V_i(x)=+\infty,
\end{equation}
for $i=1,2$ (the latter holding, of course, only for $\Omega=\R^N$).
In this paper we deal with solitary wave solutions to the following system of coupled
Gross-Pitaevskii equations:
\begin{equation}\label{eq:intro_sys}
\begin{cases}
\icomp\partial_t\Phi_1 + \Delta\Phi_1-V_1(x)\Phi_1 + ( \mu_1 |\Phi_1|^2 +\beta |\Phi_2|^2)\Phi_1=0\\
\icomp\partial_t\Phi_2 + \Delta\Phi_2-V_2(x)\Phi_2 + ( \mu_2 |\Phi_2|^2 +\beta |\Phi_1|^2)\Phi_2=0\\
\text{on }\Omega\times\R,\text{ with zero Dirichlet b.c. if }\Omega\text{ is bounded},
\end{cases}
\end{equation}
aiming at extending to systems part of the results that we obtained in a previous paper concerning
the single NLS \cite{NorisTavaresVerzini2013}.
Cubic Schr\"odinger systems like \eqref{eq:intro_sys} appear as a relevant model in different
physical contexts, such as nonlinear optics, fluid mechanics and Bose-Einstein condensation
(see for instance \cite{MR2090357,Sirakov2007} and the references provided there).
Their solutions show different qualitative behaviors depending on the sign of the
\emph{scattering lengths}
$\mu_1$, $\mu_2$, $\b$: when $\mu_i$ is positive (resp. negative), then the corresponding equation
is said to be \emph{focusing} (resp. \emph{defocusing}); when $\b$ is positive (resp. negative),
then the system is said to be \emph{cooperative} (resp. \emph{competitive}). Here we will deal
with almost any of these choices, apart from a few degenerate cases. More precisely, we will assume
that $(\mu_1,\mu_2,\beta)\in \R^3$ satisfies one of the following 
conditions:
\begin{equation}\label{eq:assumptions_nondegeneracyIntro}
\tag{NonDeg}
\begin{array}{lcl}
\mu_1 \cdot \mu_2<0& \text{ and }& \beta\in \R;\\
\mu_1,\mu_2\geq 0\text{, not both zero,}& \text{ and }& \beta\neq -\sqrt{\mu_1\mu_2};\\
\mu_1,\mu_2\leq 0\text{, not both zero,}& \text{ and }& \beta\neq \sqrt{\mu_1\mu_2}
\end{array}
\end{equation}
(although partial results can be obtained also in certain complementary cases, see some of the
remarks along this paper).

We will seek solutions to system \eqref{eq:intro_sys} among functions which belong,
at each fixed time, to the energy space
\[
\HcalC=\left\{(\Phi_1,\Phi_2): \Phi_i \in H^1_0(\Omega,\C), \ \int_\Omega
(|\nabla \Phi_i|^2+V_i(x)\Phi_i^2)\, dx<\infty,\ i=1,2\right\},
\]
endowed with its natural norm
\[
\|(\Phi_1,\Phi_2) \|^2_\Hcal = \sum_{i=1}^2\int_\Omega (|\nabla \Phi_i|^2+V_i(x)|\Phi_i|^2)\, dx.
\]
In such context, the system preserves, at least formally, both the masses
\[
\Qcal(\Phi_i)=\int_\Omega |\Phi_i|^2\, dx,\qquad i=1,2,
\]
and the energy
\[
\Ecal(\Phi_1,\Phi_2) = \frac{1}{2}\|(\Phi_1,\Phi_2)\|_\Hcal^2 - F(\Phi_1,\Phi_2),
\]
where, for shorter notation, we let
\[
F(\Phi_1,\Phi_2) = \frac{1}{4}\int_\Omega (\mu_1 |\Phi_1|^4+2 \beta |\Phi_1|^2|\Phi_2|^2
+ \mu_2 |\Phi_2|^4)\, dx.
\]

Since we work in dimension $N\leq 3$, we have that the nonlinearity is energy subcritical;
furthermore, assumption \eqref{eq:V_assumptionT} implies that the embedding
\[
\HcalC\hookrightarrow L^p(\Omega;\C^2) \text{ is compact}
\]
for every $p<2^*=2N/(N-2)$ (for every $p$ if $N\leq2$), and hence, in particular, for $p=2,4$.
On the other hand, when $N=2$ the nonlinearity is $L^2$-critical, while when $N=3$ it is
$L^2$-supercritical. Indeed, we recall that the $L^2$-critical exponent is $1+4/N$, so that
cubic nonlinearities are $L^2$-subcritical only in dimension $N=1$. In general, the
behavior of the nonlinearity with respect to the $L^2$-critical exponent has strong influence
on the dynamics, at least in the focusing case (or in the
cooperative one), see for instance the book \cite{Cazenave2003}.

Letting $\Phi_i(x,t) = e^{\icomp  \omega_i t}U_i(x)$, where $(\omega_1,\omega_2)\in\R^2$ and
$(U_1,U_2)$ belongs to
\[
\mathcal{H}=\Hcal_\R:=\left\{ (u_1,u_2) : u_i \in H^1_0(\Omega;\R), \
\int_{\Omega}(|\nabla u_i|^2+V_i(x)u_i^2) \,dx<\infty, \ i=1,2 \right\},
\]
we have that solitary waves for \eqref{eq:intro_sys} can be obtained by solving the elliptic system
\begin{equation*}
\begin{cases}
-\Delta U_1+(V_1(x)+\omega_1) U_1=\m_1 U_1^3+\b U_1 U_2^2\\
-\Delta U_2+(V_2(x)+\omega_2) U_2=\m_2 U_2^3+\b U_2 U_1^2\\
(u_1,u_2)\in\Hcal.
\end{cases}
\end{equation*}
In doing this, two different points of view are considered in the literature:
on the one hand, one can consider the
\emph{chemical potentials} $\omega_i$ as given, and search for $(U_1,U_2)$ as critical points of the
action functional
\[
\mathcal{A}_{(\omega_1,\omega_2)}(U_1,U_2) = \mathcal{E}(U_1,U_2)
- \frac{\omega_1}{2}\mathcal{Q}(U_1) - \frac{\omega_2}{2}\mathcal{Q}(U_2);
\]
on the other hand, one can take also the coefficients $\omega_i$ to be unknown.
In this latter situation, it is natural to consider the masses $\mathcal{Q}(U_i)$
as given, so that $\omega_1,\omega_2$ can be understood as Lagrange multipliers
when searching for critical points of
\[
\mathcal{E}(U_1,U_2) \qquad \text{constrained to the manifold }
\mathcal{M}:=\{(U_1,U_2):\mathcal{Q}(U_i) = m_i\},
\]
$m_1,m_2>0$ (for further comments on this alternative, we refer to the discussion
in the introduction of  \cite{NorisTavaresVerzini2013}, and references therein).

Existence issues for the cubic elliptic system above (and for its autonomous counterpart)
have attracted, in the last decade, a great interest, and a huge amount of related results is
nowadays present in the literature. Most of them are concerned with the case of fixed chemical
potentials; as a few example we quote here the papers
\cite{MR2135447,MR2263573,MR2302730,MR2372993,Sirakov2007,MR2426142,MR2563622,MR2592975,
MR2629888,MR2587453,MR2849820,MR3097257,MR2997381,soave2013existence,MR3057151},
referring to their bibliography for an extensive list of references on this topic.

On the contrary, in this paper we consider the other point of view: given positive
$m_1,m_2$,
\begin{equation}\label{eq:intro_sys_solitary}
\text{to find }(U_1,U_2,\omega_1,\omega_2)\in\Hcal\times\R^2
\text{ s.t.}
\begin{cases}
-\Delta U_1+(V_1(x)+\omega_1) U_1=\m_1 U_1^3+\b U_1 U_2^2\\
-\Delta U_2+(V_2(x)+\omega_2) U_2=\m_2 U_2^3+\b U_2 U_1^2\\
\int_\Omega U_1^2=m_1,\ \int_\Omega U_2^2=m_2.
\end{cases}
\end{equation}
Up to our knowledge, only a few papers deal with the fixed masses approach:
essentially \cite{MR2090357,MR2928850,MR2901198,MR3148108}, all of which address
the defocusing, competitive case. This case is particularly favorable, since the
energy functional $\mathcal{E}$ is coercive (a non coercive case is considered in
\cite{MR3158981}, even though for a quite different Schr\"odinger system).
On the contrary, if at least one of the
scattering lengths is positive, then $\mathcal{E}$ is no longer coercive, and the behavior
of the nonlinearity with respect to the $L^2$-critical exponent becomes crucial.
Indeed, in the subcritical case (i.e. in dimension $N=1$), the constrained functional
$\mathcal{E}|_{\mathcal{M}}$ is still coercive, and bounded below. But if $N=2,3$, then
also $\mathcal{E}|_{\mathcal{M}}$ ceases to be coercive, and it becomes not bounded below.
This is the main difficulty in searching for critical points of $\mathcal{E}|_{\mathcal{M}}$,
indeed no ``trivial'' local minima for $\mathcal{E}|_{\mathcal{M}}$ can be identified, neither a
Nehari-type manifold seems available.

Once solitary waves are obtained, a natural question regards their stability properties.
The standard notion of stability, in this framework, is that of orbital stability, which we
recall in Section \ref{sec:stab} ahead. Orbital stability for power-type
Schr\"odinger systems has been investigated in several papers, among which we
mention \cite{MR1362765,MR2676639,MR2850761,MR2901813}. It is worth remarking that these papers
are settled on the whole $\R^N$, without trapping potentials (i.e. without compact embeddings);
for this reason, they are involved only in the $L^2$-subcritical case, in which the validity
of a suitable Gagliardo-Nirenberg
inequality can be exploited. We are not aware of any paper treating stability issues for nonlinear
Schr\"odinger systems with $L^2$-critical or supercritical nonlinearity, except for some
partial application in \cite{MR1081647}.

Our strategy to obtain solutions to problem \eqref{eq:intro_sys_solitary} consists in introducing
the following auxiliary maximization problem in $\Hcal$:
\[
M(\a,\rho_1,\rho_2) = \sup \left\{ F(u_1,u_2) : \|(u_1,u_2)\|_{\mathcal{H}}^2 = \a,\
\mathcal{Q}(u_1)=\rho_1, \ \mathcal{Q}(u_2)=\rho_2 \right\},
\]
where the positive parameters $\a$, $\rho_1$, $\rho_2$ are suitably fixed. Since both $F$
and the constraints are even, possible maximum points can be chosen to have non negative components, as we
will systematically (and often tacitly) do. As a matter of fact, the problem above
leads to a new variational characterization of solutions to \eqref{eq:intro_sys_solitary}.
In turn, such characterization contains information about the orbital stability of the corresponding
solitary waves.

Coming to the detailed description of our results, let us recall that the compact embedding
$\mathcal{H}\hookrightarrow L^2$ provides the existence of the principal eigenvalues $\lambda_{V_i}$
of $-\Delta + V_i$, which are positive. Our first result reads as follows.
\begin{theorem}\label{thm:intro1}
Let $V_1,V_2$ satisfy \eqref{eq:V_assumptionT} and  $\mu_1,\mu_2,\beta$ satisfy
\eqref{eq:assumptions_nondegeneracyIntro}.
If $\rho_1,\rho_2>0$ and $\a>\lambda_{V_1}\rho_1+\lambda_{V_2}\rho_2$ then
$M(\a,\rho_1,\rho_2)$ is achieved.
Besides, for every maximum point $(u_1,u_2)$ there exists
$(\omega_1,\omega_2,\gamma)\in\R^3$, with $\gamma>0$, such that
\begin{equation}\label{eq:Uu}
(\sqrt{\gamma} u_1, \sqrt{\gamma} u_2, \omega_1, \omega_2)
\quad\text{ solves \eqref{eq:intro_sys_solitary} with }\quad
m_1 = \gamma \rho_1,\ m_2 = \gamma \rho_2.
\end{equation}
\end{theorem}
In particular, by the maximum principle, $u_1,u_2$ can be chosen to be strictly positive
in the interior of $\Omega$. The proof of Theorem \ref{thm:intro1}
is fully detailed in Section \ref{sec:varpro}, where some further properties of $M$
are also described, such as the continuity with respect to $(\a,\rho_1,\rho_2)$. In Section
\ref{sec:stab} we turn to stability issues in connection with $M$. We prove the
following criterion for stability.
\begin{theorem}\label{thm:intro2}
Under the assumptions of Theorem \ref{thm:intro1}, let $\rho_1$, $\rho_2$ be fixed
and suppose that, for some $\alpha_1<\alpha_2$, there exists a $C^1$ curve
\[
(\alpha_1,\alpha_2)\ni \quad \alpha \mapsto
(u_1(\alpha),u_2(\alpha),\omega_1(\alpha),\omega_2(\alpha),\gamma(\alpha))\quad \in\Hcal\times \R^3,
\]
such that \eqref{eq:Uu} holds, and $(u_1(\alpha),u_2(\alpha))$ achieves
$M(\a,\rho_1,\rho_2)$ for every $\a\in(\alpha_1,\alpha_2)$.
If furthermore
\[
\alpha \mapsto \gamma(\alpha)\quad\text{ is strictly increasing}
\]
then the set of solitary wave solutions to \eqref{eq:intro_sys} associated with
$M(\a,\rho_1,\rho_2)$ (according to Theorem \ref{thm:intro1}) is orbitally stable,
for every $\a\in(\alpha_1,\alpha_2)$, among solutions
which enjoy both local existence, uniformly in the $\Hcal$ norm of the initial datum, and
conservation of masses and energy.

In particular, were $M(\a,\rho_1,\rho_2)$ uniquely achieved, then the corresponding
solitary wave is conditionally stable, in the sense just explained.
\end{theorem}
The strict monotonicity of a parameter, as a condition for stability,
is reminiscent of the abstract theory developed in
\cite{GrillakisShatahStrauss,MR1081647}. In fact, our proof is inspired by
the classical paper by Shatah \cite{Shatah1983}. Observe that we only stated the
\emph{conditional} nonlinear orbital stability, where the condition is that the solution
of system \eqref{eq:intro_sys} corresponding to an initial datum $(\phi_1,\phi_2)$
exists locally in time, with the time interval uniform in $\|(\phi_1,\phi_2)\|_{\Hcal}$,
and that the masses and the energy are preserved. In fact, these properties are known to be
true for every initial datum in $\Hcal$, at least when some further restrictions about
$V_i$, $\mu_i$, $\beta$ are assumed, see for instance \cite[Chapters 3 and 4]{Cazenave2003}.
However, being the field so vast, even a rough summary of well-posedness for
Schrodinger systems with potential is far beyond the scopes of this paper. We refer the
interested reader to the entry
\href{http://wiki.math.toronto.edu/DispersiveWiki/index.php/NLS_with_potential}{``NLS with
potential''}
in the \emph{DispersiveWiki} project webpage \cite{wikidisp}
(as well as to the entries
\href{http://wiki.math.toronto.edu/DispersiveWiki/index.php/Cubic_NLS_on_R2}{``Cubic NLS
on $\R^{2}$''},
\href{http://wiki.math.toronto.edu/DispersiveWiki/index.php/Cubic_NLS_on_R3}{``Cubic NLS
on $\R^{3}$''}).

Finally, in Section \ref{sec:appl} we provide two applications of Theorems \ref{thm:intro1},
\ref{thm:intro2}, proving, in some particular cases, existence of orbitally stable solitary
wave solutions to \eqref{eq:intro_sys} having prescribed masses.

Our first application deals with the case of small masses. In Section \ref{subsec:AP} we prove
the following.
\begin{theorem}\label{thm:intro3}
Let assumptions \eqref{eq:V_assumptionT}, \eqref{eq:assumptions_nondegeneracyIntro} hold.
For every $k\geq 1$ there exists $\bar m>0$, such that for every positive $m_1,m_2$ satisfying
\[
\frac{1}{k}\leq \frac{m_2}{m_1} \leq{k},\qquad
m_1+m_2 \leq \bar m,
\]
there exists $(U_1,U_2,\omega_1,\omega_2)\in\Hcal\times\R^2$, with $U_i$ positive in
$\Omega$, solution to \eqref{eq:intro_sys_solitary}. Furthermore, the corresponding solitary wave
\[
(\Phi_1(x,t),\Phi_2(x,t)) = (e^{\icomp  \omega_1 t}U_1(x),e^{\icomp  \omega_2 t}U_2(x))
\]
is conditionally orbitally stable for system \eqref{eq:intro_sys},
in the sense of Theorem \ref{thm:intro2}.
\end{theorem}
We remark that, apart from condition \eqref{eq:assumptions_nondegeneracyIntro}, no restriction
about $\mu_1$, $\mu_2$ and $\b$ is required. In order to prove such theorem, we will exploit a
parametric version of a classical result by Ambrosetti and Prodi \cite{AmbrosettiProdipaper}
about the inversion of maps with
singularities, see Theorem \ref{thm:AmbrosettiProdi_parametric} below. In particular, we
rely on the fact that, if $m_1/m_2$ is fixed, our problem can be reduced to an inversion
of a map near an ordinary singular point, while this property is lost if one of the masses
vanish. This is the reason for the restriction on $m_1/m_2$. On the other hand, when
one mass vanishes, the system reduces to a single equation: since we already treated
successfully this case in \cite{NorisTavaresVerzini2013}, it is presumable that the result
should hold without such restriction.

As a last application, in Section \ref{sec:defocusing} we deal with the case of defocusing,
weakly interacting systems, meaning that $\mu_1,\mu_2$ are negative and $\b^2<\mu_1\mu_2$.
In such case, Theorems \ref{thm:intro1} and \ref{thm:intro2} provide, for every choice of
the masses $m_1,m_2$, the existence of a unique solitary wave, and its stability,
see Theorem \ref{thm:defocusing_weak_interaction} below.
As we mentioned, in this case $\mathcal{E}$ is coercive and bounded below, so that existence
can be obtained also by the direct method, as already done in \cite{MR2090357}.
For the same reason, stability is
somewhat expected, even though it can not be obtained directly, due to the lack of
a suitable Gagliardo-Nirenberg inequality in dimension $N=2,3$.

As a final remark, let us mention that in the proofs of Theorems \ref{thm:intro1} and
\ref{thm:intro2} we use the compact embedding $\mathcal{H}\hookrightarrow L^p$
just to pass from weak to strong convergence, for maximizing sequences associated
to $M$. In the relevant case $\Omega=\R^N$, $V_i\equiv 1$, such compactness does not
hold, but one could try to adapt the same strategy by using a concentration-compactness type
argument. In conclusion, it is our belief that Theorems \ref{thm:intro1} and
\ref{thm:intro2} should hold in a more general situation, however this falls out
of the scopes of the present paper.

\subsection*{Notations and preliminaries}
In the following, we will say that a pair $(u_1,u_2)$ is positive (nonnegative) if both
$u_1$ and $u_2$ are. We remark that, whenever $\mathcal{Q}(u_1)$, $\mathcal{Q}(u_2)$ are
fixed to be positive, then both trivial and semitrivial pairs are excluded.

As we already noticed, the embedding $\mathcal{H}\hookrightarrow L^p$ is compact, for $p$
Sobolev subcritical. In turn,
the compact embedding implies the existence of a first eigenvalue. In the following we denote
by $\varphi_{V_i}$ the unique nonnegative function which achieves
\[
\lambda_{V_i}=\inf \left\{ \int_\Omega \left(|\nabla\varphi|^2+V_i(x) \varphi^2\right)\,dx:\ \int_
\Omega \varphi^2\,dx=1 \right\}.
\]
We remark that $\lambda_{V_i}>0$ by assumption \eqref{eq:V_assumptionT} (in fact, the positivity
assumption there may be replaced by the requirement that $V_i$ is bounded from below, by
performing a change of gauge $\Phi_i \rightsquigarrow \Phi_i \exp[\icomp t\inf V_i]$).
In such arguments, the compactness of the embedding is immediate if $\Omega$ is bounded;
in case $\Omega=\R^N$, it can be obtained in a rather standard way, for instance mimicking the
proof of \cite[Proposition 6.1]{MR1260616}, which is performed in the particular case $V_i(x)=|x|^2$.

Throughout the paper, ``$\icomp$'' indicates the imaginary unit, while $i$ and $j$ stand for indexes
between $1$ and $2$, with $j\neq i$.
Finally, we denote with $C$ any positive constant we need not to specify, which may change its value
even within the same expression.

\section{A variational problem}\label{sec:varpro}
Throughout this section, $\mu_1,\mu_2, \beta$ satisfy assumption
\eqref{eq:assumptions_nondegeneracyIntro} while $V_1,V_2$ satisfy
assumption \eqref{eq:V_assumptionT}. For $(u_1,u_2)\in\mathcal{H}$, recall that
\[
F(u_1,u_2)=\int_\Omega \left( \mu_1 \frac{u_1^4}{4}+\beta \frac{u_1^2 u_2^2}{2} +\mu_2 \frac{u_2^4}{4}
\right) \,dx.
\]
We consider the following maximization problem
\begin{equation}\label{eq:M_alpha}
M(\a,\rho_1,\rho_2)=\sup_{\mathcal{U}(\a,\rho_1,\rho_2)} F(u_1,u_2)
\end{equation}
where, for  $\rho_1,\rho_2>0$ and $\a\geq\lambda_{V_1}\rho_1+\lambda_{V_2}\rho_2$, we define
\[
\mathcal{U}(\a,\rho_1,\rho_2)=\left\{ (u_1,u_2)\in \mathcal{H}: \begin{array}{ll}
\|(u_1,u_2)\|_{\mathcal{H}}^2 \leq\a, \\[0.1cm]
\int_\Omega u_i^2\,dx=\rho_i, \ i=1,2
\end{array}
\right\}.
\]
As we will see in a moment, under assumption \eqref{eq:assumptions_nondegeneracyIntro}, this
definition of $M$  is equivalent to the one given in the introduction.
\begin{remark}\label{rem:trivial_constraint}
For $\a = \lambda_{V_1}\rho_1+\lambda_{V_2}\rho_2$ we have that
\[
\mathcal{U}(\a,\rho_1,\rho_2)=\left\{ ((-1)^l\sqrt{\rho_1}\vphi_{V_1},
(-1)^m\sqrt{\rho_2}\vphi_{V_2}) : (l,m)\in\{0,1\}^2\right\},
\]
thus $F$ is constant in $\mathcal{U}$ and $M$ is trivially achieved. Of course, if
$\a < \lambda_{V_1}\rho_1+\lambda_{V_2}\rho_2$ then $\mathcal{U}$ is empty.
\end{remark}
\begin{lemma}\label{lemma:U_alpha_manifold}
For every $\a>\lambda_{V_1}\rho_1+\lambda_{V_2}\rho_2$, the set
\[
\tilde{\mathcal{U}}(\a,\rho_1,\rho_2)=\left\{ (u_1,u_2)\in \mathcal{U}(\a,\rho_1,\rho_2): \begin{array}{ll}
\|(u_1,u_2)\|_{\mathcal{H}}^2=\a, \\[0.1cm]
\int_\Omega u_i\varphi_{V_i}\,dx\neq0 \ i=1,2
\end{array}
\right\}
\]
is a submanifold of $\mathcal{H}$ of codimension 3.
\end{lemma}
\begin{proof}
It is easy to see that $\tilde{\mathcal{U}}$ is not empty. Letting
\[
G(u_1,u_2)=\left(\int_\Omega u_1^2\,dx-\rho_1, \, \int_\Omega u_2^2\,dx-\rho_2, \,
\|(u_1,u_2)\|_{\mathcal{H}}^2-\alpha \right),
\]
it suffices to prove that for every $u\in\tilde{\mathcal{U}}(\a,\rho_1,\rho_2)$ the range of
$G'(u_1,u_2)$ is $\R^3$. This can be checked by evaluating $G'(u_1,u_2)[\phi_1,\phi_2]$ with
$(\phi_1,\phi_2)$ equal to $(u_1,u_2)$, $(\varphi_{V_1},0)$ and $(0,\varphi_{V_2})$ respectively, and
recalling that $\a\neq\lambda_{V_1}\rho_1+\lambda_{V_2}\rho_2$.
\end{proof}

\begin{lemma}\label{lemma:maximizer_solves_sys}
For every $\a\geq\lambda_{V_1}\rho_1+\lambda_{V_2}\rho_2$ \eqref{eq:M_alpha} is achieved.
Moreover, every maximum $(u_1,u_2)$ belongs to $\tilde{\mathcal{U}}(\a,\rho_1,\rho_2)$, and there exist
$\omega_1,\omega_2,\gamma\in\R$ such that
\begin{equation}\label{eq:system_main}
-\Delta u_i+(V_i(x)+\omega_i) u_i=\gamma(\m_i u_i^3+\b u_i u_j^2), \quad i=1,2,\, j\neq i.
\end{equation}
\end{lemma}
\begin{proof}
If $\a = \lambda_{V_1}\rho_1+\lambda_{V_2}\rho_2$ then by Remark \ref{rem:trivial_constraint}
the result immediately follows by choosing $\omega_i=-\lambda_{V_i}$, $\gamma=0$.

Otherwise, it is not difficult to see that $M(\a,\rho_1,\rho_2)$ is achieved by a couple
$(u_1,u_2)\in \mathcal{U}(\a,\rho_1,\rho_2)$. Indeed, $\mathcal{U}(\a,\rho_1,\rho_2)$ is not empty
and weakly compact in $\mathcal{H}$, $F(u_1,u_2)$ is weakly continuous and bounded in
$\mathcal{U}(\a,\rho_1,\rho_2)$:
\[
\left| F(u_1,u_2) \right| \leq C (|\m_1|+|\m_2|+|\b|)\a^2.
\]
By possibly taking $|u_i|$ we can suppose $u_i\geq0$.

Suppose in view of a contradiction that the maximizer does not belong to $\tilde{\mathcal{U}}(\a,\rho_1,\rho_2)$, i.e. $\|(u_1,u_2)\|_{\mathcal{H}}^2<\alpha$ .
 Then there exist two Lagrange multipliers $\omega_1,\omega_2$ such that almost everywhere we have
\begin{equation}\label{eq:contradiction_u_i_const}
\m_1 u_1^3+\b u_1 u_2^2=\omega_1 u_1 \quad\text{and}\quad \b u_1^2 u_2 + \m_2 u_2^3 =\omega_2u_2.
\end{equation}
a) \underline{If $\b^2\neq \m_1\m_2$:} this implies that the $u_i$ are piecewise constant;
since $u_i\in H^1_0(\Omega)$, $u_i\not\equiv0$, we have reached a contradiction. \\
b) \underline{The remaining cases} $\mu_1,\mu_2>0$, $\beta=\sqrt{\mu_1\mu_2}$ and $\mu_1,\mu_2<0$,
$\beta=-\sqrt{\mu_1\mu_2}$ are much more delicate, and we will analyze them in detail during the
remainder of the proof. First of all, we claim that $\omega_1=\omega_2=0$. To start with, suppose
that $\Omega$ is bounded, Consider the extension of $u_i$ to the whole $\R^N$ by 0, denoting it also
by $u_i$. With this notation, $u_i\in H^1(\R^N)$, hence by \cite[Remark 3.3.5]{Ziemer} we have that
each $u_i$ is \emph{approximately continuous}, this meaning that for $\mathcal{H}^{N-1}$--a.e.
$x_0\in \R^N$ there exists a measurable set $A_{x_0}^i$ such that
\begin{equation}\label{eq:auxiliary_appcontinuity}
\lim_{r\to 0}\frac{|A_{x_0}^i \cap B_r(x_0)|}{|B_r(x_0)|}=1, \qquad
u_i|_{A_{x_0}^i} \text{ is continuous at } x_0.
\end{equation}
Observe that clearly $|A_{x_0}^1\cap A_{x_0}^2 \cap B_r(x_0)|/|B_r(x_0)|\to 1$ as well. Thus, as $\Omega$ is Lipschitz (and hence $\mathcal{H}^{N-1}(\partial\Omega)>0$) and $u_i=0$ on $\partial \Omega$, there exist $x_0\in \overline \Omega$ with $u_1(x_0)=u_2(x_0)=0$, $A_{x_0}^i$ satisfying \eqref{eq:auxiliary_appcontinuity}, and $x_n\in A_{x_0}^1\cap A_{x_0}^2\cap \Omega$ converging to $x_0$, such that either $u_1(x_n)\neq 0$ or $u_2(x_n)\neq 0$.
\begin{itemize}
\item If $u_1(x_n),u_2(x_n)\neq 0$, then \eqref{eq:contradiction_u_i_const} implies that
\[
\m_1u_1^2(x_n)+\b u_2^2(x_n)=\omega_1,\qquad \m_2u_2^2(x_n)+\b u_1^2(x_n)=\omega_2,
\]
and thus (by making $n\to \infty$) we have $\omega_1=\omega_2=0$.

\item If $u_1(x_n)\neq 0$ and $u_2(x_n)=0$, then from \eqref{eq:contradiction_u_i_const} we have that $u_1^2(x_n)=\omega_1/\m_1$, and thus $\omega_1=0$, a contradiction. Reasoning in an analogous way, the case $u_1(x_n)=0$ and $u_2(x_n)\neq0$ also leads to a contradiction.
\end{itemize}
Thus we have proved that $\omega_1=\omega_2=0$ in the case $\Omega$ is bounded.
If $\Omega=\R^N$ we can reason in a similar way. By \eqref{eq:V_assumptionT} we have that every $(u_1,u_2)\in\mathcal{H}$ satisfies
\begin{equation}\label{eq:decay_u}
\lim_{R\to\infty} \int_{\R^N\setminus B_R} u_i^2 \,dx=0, \qquad  i=1,2.
\end{equation}
Hence for every $\eps>0$ there exist $x_\eps$ with $0<u_1(x_\eps)^2+u_2(x_\eps)^2\leq\eps$ and
$A^1_{x_\eps},A^2_{x_\eps}$ satisfying \eqref{eq:auxiliary_appcontinuity}. Proceeding as above,
since $\eps$ is arbitrary, we obtain $\omega_1=\omega_2=0$.
Therefore we have proved that \eqref{eq:contradiction_u_i_const} writes as
\begin{equation}\label{eq:contradiction_omega_i_zero}
\m_i u_i^2 + \b u_j^2=0 \quad i,j=1,2, \ j\neq i,
\end{equation}
a.e. in $\Omega$. This, in turn, implies $\m_i\rho_i+\b\rho_j=0$, which provides a contradiction
also in case b).

In conclusion, we have shown that the maximizer $(u_1,u_2)$ belongs to $\tilde{\mathcal{U}}_\a$. By Lemma \ref{lemma:U_alpha_manifold} the Lagrange multipliers theorem applies. Since we have shown in addition that $(u_1,u_2)$ can not satisfy \eqref{eq:contradiction_u_i_const}, we conclude that it satisfies \eqref{eq:system_main}.
\end{proof}

\begin{lemma}\label{lemma:gamma>0}
Given $\a>\lambda_{V_1}\rho_1+\lambda_{V_2}\rho_2$, let $(u_1,u_2)\in \tilde{\mathcal{U}}(\a,\rho_1,\rho_2)$ achieve \eqref{eq:M_alpha}. Then in \eqref{eq:system_main} we have $\gamma>0$.
\end{lemma}
\begin{proof}
We proceed similarly to \cite[Prop. 2.4]{NorisTavaresVerzini2013}.
For $i=1,2$ and $t\in\R$ close to $1$, let
\[
w_i(t)=tu_i+s_i(t)\sqrt{\rho_i} \vphi_{V_i},
\]
where $s_i(t)$ are such that
\begin{equation}\label{eq:s_i}
\rho_i=\int_\Om w_i(t)^2\,dx=t^2\rho_i +2t s_i(t)\sqrt{\rho_i} \int_\Om u_i\vphi_{V_i}\,dx+s_i(t)^2\rho_i, \quad s_i(1)=0.
\end{equation}
Since
\[
\left.\partial_{s_i}\left(t^2\rho_i+2ts_i\sqrt{\rho_i}\int_\Om u_i\vphi_{V_i}\,dx+s_i^2\rho_i\right)\right|_{(t,s)=(1,0)}=2\sqrt{\rho_i} \int_\Om u_i\vphi_{V_i}\,dx\neq0,
\]
the Implicit Function Theorem applies, providing that the maps $t\mapsto w_i(t)$ are of class $C^1$ in a neighborhood of $t=1$.
The first relation in \eqref{eq:s_i} provides
\[
0=\int_\Omega u_iw_i'(1)\,dx=\rho_i+s_i'(1)\sqrt{\rho_i}\int_\Omega u_i\varphi_{V_i}\,dx.
\]
Therefore $s_i'(1)=-\sqrt{\rho_i}/\int_\Omega u_i\vphi_{V_i}\,dx$ and
$w_i'(1)=u_i-(\rho_i/\int_\Omega u_i\vphi_{V_i}\,dx)\vphi_{V_i}$.
We use the last estimates to compute
\[
\begin{split}
\left.\frac{1}{2}\frac{d}{dt} \int_\Omega \left( |\nabla w_i(t)|^2+V_i(x)w_i(t)^2\right)\, dx\right|_{t=1}
&= \int_\Om\left(\nabla u_i\cdot\nabla w'_i(1)+V_i(x)u_iw'_i(1)\right)\,dx \\
&= \int_\Omega \left(|\nabla u_i|^2+V_i(x)u_i^2\right)\, dx- \rho_i\lambda_{V_i},
\end{split}
\]
and hence
\begin{equation}\label{eq:sign_of_gamma1}
\left.\frac{1}{2}\frac{d}{dt} \|( w_1(t),w_2(t) )\|^2_{\mathcal{H}}\right|_{t=1}
=\alpha-(\lambda_{V_1}\rho_1+\lambda_{V_2}\rho_2)>0.
\end{equation}
Thus there exists $\eps>0$ such that $(w_1(t),w_2(t))\in\mathcal{U}(\a,\rho_1,\rho_2)$ for $t\in (1-\eps, 1]$.
Since $(w_1(1),w_2(1))=(u_1,u_2)$ achieves the maximum of $F$ in $\mathcal{U}(\a,\rho_1,\rho_2)$, we deduce
\begin{equation}\label{eq:sign_of_gamma2}
\left.\frac{d}{dt} F(w_1(t),w_2(t)) \right|_{t=1}\geq 0.
\end{equation}
On the other hand, using  \eqref{eq:system_main} and the fact that $\int_\Omega u_iw_i'(1)\,dx=0$, we have
\[
\begin{split}
\left.\gamma \frac{d}{dt} F(w_1(t),w_2(t)) \right|_{t=1} \\
= \int_\Omega \left[ (-\Delta u_1+V_1(x)u_1) w_1'(1)+(-\Delta u_2+V_2(x)u_2) w_2'(1) \right]\,dx \\
= \left. \|( w_1(t),w_2(t) )\|^2_{\mathcal{H}} \right|_{t=1}.
\end{split}
\]
By comparing the last relation with \eqref{eq:sign_of_gamma1} and \eqref{eq:sign_of_gamma2} we obtain
the statement.
\end{proof}
We are ready to prove our first main result.
\begin{proof}[Proof of Theorem \ref{thm:intro1}]
By Lemma \ref{lemma:maximizer_solves_sys}, for any
$(u_1,u_2)\in \arg\max M(\alpha,u_1,u_2)$ (which is not empty) there exists
$(\omega_1,\omega_2,\gamma)\in\R^3$, such that \eqref{eq:system_main} holds. Moreover, since by assumption $\a>\lambda_{V_1}\rho_1
+\lambda_{V_2}\rho_2$, Lemma \ref{lemma:gamma>0} implies that $\gamma>0$.
The only thing that remains to prove is that \eqref{eq:Uu} holds. This is a direct
consequence of \eqref{eq:system_main} since, setting $U_i = \sqrt{\gamma} u_i$, we obtain
\begin{multline*}
-\Delta U_i+(V_i(x)+\omega_i) U_i=
\gamma^{1/2}\left(-\Delta u_i+(V_i(x)+\omega_i) u_i\right)\\
=\gamma^{3/2}(\m_i u_i^3+\b u_i u_j^2)
=\m_i U_i^3+\b U_i U_j^2. \qedhere
\end{multline*}
\end{proof}
In the remainder of this section we will prove some properties of $M$ and of system
\eqref{eq:system_main} which we will use later on.
A remarkable property is that $M$ is a continuous function.
\begin{lemma}\label{lem:M_is_continuous}
Let $(\a_n,\rho_{1,n},\rho_{2,n})\to(\bar\a,\bar\rho_{1},\bar\rho_{2})$, with
$\a_n \geq \lambda_{V_1}\rho_{1,n}+\lambda_{V_2}\rho_{2,n}$. Then
\[
M(\a_n,\rho_{1,n},\rho_{2,n})\to M(\bar\a,\bar\rho_{1},\bar\rho_{2}).
\]
\end{lemma}
\begin{proof}
a) \underline{$\limsup M(\a_n,\rho_{1,n},\rho_{2,n})\leq M(\bar\a,\bar\rho_{1},\bar\rho_{2})$.}
Indeed, let $(u_{1,n},u_{2,n})\in\tilde{\mathcal{U}}(\a_n,\rho_{1,n},\rho_{2,n})$ achieve
$M(\a_n,\rho_{1,n},\rho_{2,n})$. Since $\alpha_n$ is bounded, we deduce that
$(u_{1,n},u_{2,n})$ converges (up to subsequences) weakly in $\Hcal$ to some $(u_1^*,u_2^*)$. By the compact embedding, $(u_1^*,u_2^*) \in
\mathcal{U}(\bar\a,\bar\rho_{1},\bar\rho_{2})$ and
\[
M(\a_n,\rho_{1,n},\rho_{2,n}) \to F(u_1^*,u_2^*)\leq M(\bar\a,\bar\rho_{1},\bar\rho_{2}).
\]
b) \underline{$\liminf M(\a_n,\rho_{1,n},\rho_{2,n})\geq M(\bar\a,\bar\rho_{1},\bar\rho_{2})$.}
We assume $\bar\a > \lambda_{V_1}\bar\rho_1+\lambda_{V_2}\bar\rho_2$, the complementary case
being an easy consequence of Remark \ref{rem:trivial_constraint}. Let $(\bar u_{1},\bar u_{2})\in
\tilde{\mathcal{U}}(\bar\a,\bar\rho_{1},\bar\rho_{2})$, with
non negative components, achieve $M(\bar\a,\bar\rho_{1},\bar\rho_{2})$. To conclude, we will
construct a sequence $(w_{1,n},w_{2,n})\in
\tilde{\mathcal{U}}(\a_n,\rho_{1,n},\rho_{2,n})$ in such a way that $(w_{1,n},w_{2,n})\to (\bar u_{1},
\bar u_{2})$, strongly in $\Hcal$. Indeed, this would imply
\[
M(\a_n,\rho_{1,n},\rho_{2,n})\geq F(w_{1,n},w_{2,n}) \to
F(\bar{u}_1,\bar{u}_2)= M(\bar\a,\bar\rho_{1},\bar\rho_{2}).
\]
Since $\bar\a > \lambda_{V_1}\bar\rho_1+\lambda_{V_2}\bar\rho_2$,
we can assume without loss of generality that
\begin{equation}\label{eq:continuity}
\int_\Omega (|\nabla \bar u_1 |^2+V_1(x)\bar u_1^2)\,dx > \lambda_{V_1} \bar \rho_1.
\end{equation}
Taking
\[
w_1(a,b) = (1+a)\bar{u}_1 + b \vphi_{V_1},\qquad w_2 (c) = (1+c)\bar{u}_2
\]
our task is reduced to apply the Inverse Function Theorem to the map
\[
f(a,b,c) = \left(\|(w_1(a,b),w_2 (c))\|^2_{\Hcal},\|w_1(a,b)\|^2_{L^2},\|w_2 (c)\|^2_{L^2}\right)
\]
near $f(0,0,0)=(\bar\a,\bar\rho_{1},\bar\rho_{2})$. A direct calculation yields
\begin{multline*}
\det f'(0,0,0) = 8\int_\Omega \bar u_2^2\,dx \left[
\int_\Omega (|\nabla\bar u_1|^2 + V_1(x)\bar u_1^2)\,dx \cdot
\int_\Omega \bar u_1\vphi_{V_1}\,dx\right.\\ -\left.
\int_\Omega (\nabla \bar u_1\cdot \nabla \vphi_{V_1} + V_1(x) \bar u_1\vphi_{V_1})\,dx \cdot
\int_\Omega \bar u_1^2 \,dx
\right]\\
= 8\bar\rho_2 \int_\Omega \bar u_1\vphi_{V_1}\,dx \left[
\int_\Omega (|\nabla\bar u_1|^2 + V_1(x)\bar u_1^2)\,dx - \lambda_{V_1} \bar \rho_1
\right],
\end{multline*}
which is positive by \eqref{eq:continuity}.
\end{proof}
\begin{corollary}\label{coro:strong_conv}
Let $(u_{1,n},u_{2,n})\in\tilde{\mathcal{U}}(\a_n,\rho_{1,n},\rho_{2,n})$ achieve
$M(\a_n,\rho_{1,n},\rho_{2,n})$. If $(\a_n,\rho_{1,n},\rho_{2,n})\to(\bar\a,\bar\rho_{1},
\bar\rho_{2})$ then, up to subsequences,
\[
(u_{1,n},u_{2,n}) \to (\bar u_1, \bar u_2) \qquad\text{achieving }M(\bar\a,\bar\rho_{1},
\bar\rho_{2}),
\]
the convergence being strong in $\Hcal$. Indeed, once weak convergence to a maximizer has
been obtained, then Lemma \ref{lemma:maximizer_solves_sys} implies that
$\|(\bar u_1, \bar u_2)\|^2_{\Hcal} = \bar \alpha = \lim_n \alpha_n =
\lim_n\|(u_{1,n},u_{2,n})\|^2_{\Hcal}$.
\end{corollary}
As one may suspect, the convergence of the maxima and that of the maximizers implies the
one of the Lagrange multipliers appearing in \eqref{eq:system_main}. As a matter
of fact, this holds even in more general situations, as we show in the following lemma.
\begin{lemma}\label{lemma:bounded_quantities}
Take a sequence $(u_{1,n},u_{2,n},\omega_{1,n},\omega_{2,n},\gamma_n)$ such that
\begin{equation*}
\begin{cases}
-\Delta u_{1,n}+(V_1(x)+\omega_{1,n}) u_{1,n}=\gamma_n(\mu_1u_{1,n}^3+\beta u_{1,n}u_{2,n}^2)\\
-\Delta u_{2,n}+(V_2(x)+\omega_{2,n}) u_{2,n}=\gamma_n(\mu_2u_{2,n}^3+\beta u_{2,n}u_{1,n}^2)\\[0.2cm]
\int_\Omega u_{1,n}^2\, dx=\rho_{1,n},\quad \int_\Omega u_{2,n}\, dx=\rho_{2,n},
\end{cases}
\end{equation*}
and assume that
\[
\rho_{1,n},\ \rho_{2,n}\ \text{ and } \ \|(u_{1,n},u_{2,n})\|_\Hcal^2=:\alpha_n \qquad \text{ are bounded}
\]
both from above, and from below, away from zero.
Then the sequences $\omega_{1,n}$, $\omega_{2,n}$, $\gamma_n$ are bounded.
\end{lemma}
\begin{proof}
Take $u_i$ such that $u_{i,n}\rightharpoonup u_i$ weakly in $\Hcal$, strongly in $L^p(\Omega)$,
$1<p<2^*$, and let $\int_\Omega u_i^2\,dx =: \rho_i$.

\noindent a) \underline{$\omega_{i,n}$ are bounded.} Suppose, in view of a contradiction, that $|\omega_{1,n}|\to \infty$. By multiplying the equation for $u_{1,n}$ by $u_{1,n}$ itself and dividing the result by $\omega_{1,n}$, we obtain
\[
\frac{1}{\omega_{1,n}}\int_\Omega (|\nabla u_{1,n}|^2+V_1(x)u_{1,n}^2)\, dx+\rho_{1,n}
=\frac{\gamma_n}{\omega_{1,n}}\int_\Omega (\m_1u_{1,n}^4+\b u_{1,n}^2 u_{2,n}^2)\, dx.
\]
As $\alpha_n$ is bounded, by taking the limit in $n$, it holds
\[
\rho_1=A\int_\Omega (\m_1u_1^4+\b u_1^2u_2^2)\, dx
\]
where $\lim_n \frac{\gamma_n}{\omega_{1,n}}=:A\neq 0$ (which also implies that $\gamma_n\to +\infty$). Going back to the first equation, multiplying it by an arbitrary test function $\phi$, dividing the result by $\omega_{1,n}$, and passing to the limit, we see that
\[
\int_\Omega u_1 \phi\, dx=A \int_\Omega (\m_1u_1^3 +\b u_1u_2^2 )\phi\, dx,
\]
and hence, since $u_1>0$ in $\Omega$ (by the maximum principle) we have the pointwise identity
\[
1=A(\m_1 u_1^2+\b u_2^2).
\]
As the trace of $u_1$ and $u_2$ is zero on $\partial \Omega$, we obtain a contradiction and thus
$\omega_{1,n}$ is a bounded sequence. The case $\omega_{2,n}$ unbounded can be ruled out in an
analogous way.
\smallbreak

\noindent b) \underline{$\gamma_n$ is bounded.}
Assume by contradiction that $\gamma_n\to +\infty$. Multiplying the $i$--th equation by any test function $\phi$, integrating by parts, dividing the result by $\gamma_n$ and passing to the limit, at the end we deduce that
\[
\m_1u_1^2+\b u_2^2=0,\qquad \m_2 u_2^2+\b u_1^2=0.
\]
Furthermore, the integration of these two equations yields the identities
\[
\m_1\rho_1+\b\rho_2=0,\qquad \m_2\rho_2+\b\rho_1=0.
\]
This clearly is a contradiction if $(\mu_1,\mu_2,\beta)$ satisfies
\eqref{eq:assumptions_nondegeneracyIntro}.
\end{proof}
To conclude this section, we give some hint of the kind of problems which
arise in case assumption \eqref{eq:assumptions_nondegeneracyIntro} does not hold.
\begin{remark}
When \eqref{eq:assumptions_nondegeneracyIntro} does not hold there are specific conditions
about $\rho_1$, $\rho_2$ which allow to develop the above theory in some cases. On
the other hand, in general, degenerate situations may appear.

For instance, if $\m_1,\m_2<0$ and $\b=\sqrt{\m_1\m_2}$, then
\[
F(u_1,u_2) = -\frac{|\m_1|}{4}\int_\Omega\left( u_1^2-\frac{\sqrt{|\m_2|}}{\sqrt{|\m_1|}}u_2^2
\right)^2\,dx  \leq 0;
\]
if furthermore $\sqrt{|\m_1|}\rho_1 = \sqrt{|\m_2|}\rho_2$, then
\[
F(\rho_1 \psi , \rho_2 \psi) = 0\qquad \text{for every }\psi.
\]
Choosing $\psi$ as the eigenfunction achieving
\[
\hat\alpha=\inf \left\{ \int_\Omega \left(|\nabla\psi|^2+\frac{\rho_1V_1(x) + \rho_2 V_2(x)
}{\rho_1+\rho_2} \psi^2\right)\,dx:\ \int_\Omega \psi^2\,dx=1 \right\},
\]
then $M_\alpha=0$ is attained by $(\sqrt{\rho_1} \psi , \sqrt{\rho_2} \psi)$ for every
$\alpha\geq\hat\alpha(\rho_1+\rho_2)$,
but it belongs to $\tilde{\mathcal{U}}_\a$ only for $\a=\hat\a$. Moreover, if $V_1=V_2=V$, then
$\psi=\vphi_{V}$, and
$(\sqrt{\rho_1}\varphi_{V},\sqrt{\rho_2}\varphi_V,-\lambda_{V},-\lambda_{V},\gamma)$ is a solution of
\eqref{eq:system_main} for \emph{every} $\gamma>0$.
\end{remark}

\section{A general stability result}\label{sec:stab}

Let us fix $(\alpha^*, \rho_1^*,\rho_2^*)$ such that Theorem \ref{thm:intro1} holds.
In this section we will show that, if for $\alpha$ near $\a^*$ the maximum points corresponding
to $M(\alpha, \rho_1^*,\rho_2^*)$ are along a smooth curve, with the multiplier $\gamma$ increasing
with respect to $\alpha$, then the corresponding solitary waves are conditionally orbitally
stable for an associated Schr\"odinger system. As a byproduct, we will obtain the proof of Theorem
\ref{thm:intro2}.

To be precise, let us consider the following conditions:
\begin{enumerate}
\item[(M1)]\label{item:M1} $M(\alpha^*,\rho_1^*,\rho_2^*)$ is achieved by a unique positive pair
$(u_1^*,u_2^*)$.
\item[(M2)]\label{item:M2} There exists an interval $(\alpha_1,\alpha_2)$ containing $\alpha^*$
and a $C^1$ curve
\[
(\alpha_1,\alpha_2)\to \Hcal\times \R^3,\qquad \alpha\mapsto (u_1(\alpha),u_2(\alpha),\omega_1(\alpha),\omega_2(\alpha),\gamma(\alpha))
\]
such that $(u_1(\alpha^*),u_2(\alpha^*))=(u_1^*,u_2^*)$ and
\[
\begin{cases}
(u_1(\alpha),u_2(\alpha))\text{ achieves } M(\alpha,\rho_1^*,\rho_2^*),\\
-\Delta u_i+(V_i(x)+\omega_i)u_i=\gamma(\mu_i u_i^3+\beta u_i u_j^2)\qquad i=1,2,\ j\neq i,
\end{cases}
\]
for every $\alpha\in (\alpha_1,\alpha_2)$ (recall Lemma \ref{lemma:maximizer_solves_sys}).
\item[(M3)] \label{item:M3}The map
\[
(\alpha_1,\alpha_2)\to \R,\qquad \alpha \mapsto \gamma(\alpha)
\]
is strictly increasing.
\end{enumerate}
For easier notation, let us write $\omega_i^* = \omega_i(\a^*)$, $\gamma^*=\gamma(\a^*)$.
Take the NLS system:
\begin{equation}\label{eq:NLS_with_gamma}
\begin{cases}
\icomp \partial_t \Psi_i+\Delta \Psi_i-V_i(x) \Psi_i + \gamma^* ( \mu_i |\Psi_i|^2 +\beta
|\Psi_j|^2)\Psi_i=0 ,\\
\Psi_i(0)=\psi_i,\ i=1,2,\qquad (\psi_1,\psi_2)\in \HcalC.
\end{cases}
\end{equation}
Associated to this system, we have the energy
\[
\Ecal_{\gamma^*}(\Psi_1,\Psi_2)=\frac{1}{2}\|(\Psi_1,\Psi_2)\|_\Hcal^2-\gamma^* F(\Psi_1,\Psi_2)
\]
and the masses $\Qcal(\Psi_i)=\int_{\Omega}|\Psi_i|^2\,dx$, $i=1,2$.
For \eqref{eq:NLS_with_gamma}, we assume the following local well posedness property.
\begin{enumerate}
\item[(LWP)]\label{item:H} We have local existence for \eqref{eq:NLS_with_gamma},
locally in time and uniformly in $\|(\psi_1,\psi_2)\|_{\Hcal}$. Moreover, the energy and
the masses are conserved along trajectories, that is
\[
\Ecal_{\gamma^*}(\Psi_1(t),\Psi_2(t))=\Ecal_{\gamma^*}(\psi_1,\psi_2)\quad \text{ and }\quad
\Qcal(\Psi_i(t))=\Qcal(\psi_i) \ \text{ for } i=1,2
\]
for every existence time.
\end{enumerate}

Let us recall the notion of orbital stability for the NLS system.
\begin{definition}
A standing wave solution $(e^{\icomp t \omega_1}u_1,e^{\icomp t \omega_2}u_2)$ is called
\emph{orbitally stable} for \eqref{eq:NLS_with_gamma} if for each $\eps>0$ there
exists $\delta>0$ such that,
whenever $(\psi_1,\psi_2)\in \HcalC$ satisfies $\|(\psi_1,\psi_2)-(u_1,u_2)\|_{\Hcal}<
\delta$ and $(\Psi_1(t,x), \Psi_2(t,x))$ solves \eqref{eq:NLS_with_gamma} in some
interval $[0,T_0)$, then
\begin{equation}\label{eq:condition_stability1}
(\Psi_1(t),\Psi_2(t))\text{ can be continued to a solution in } 0\leq t<\infty,
\end{equation}
and
\begin{equation}\label{eq:condition_stability2}
\sup_{0\leq t<\infty} \inf_{s_1,s_2\in \R} \| (\Psi_1(t),\Psi_2(t))-(e^{\icomp s_1}u_1,
e^{\icomp s_2}u_2)\|_{\Hcal}<\eps.
\end{equation}
\end{definition}
The purpose and main result of this section is to prove the following stability criterion.
\begin{theorem}\label{thm:stability_general}
Let $\mu_1,\mu_2,\b$ satisfy assumption \eqref{eq:assumptions_nondegeneracyIntro} and
$V_1,V_2$ satisfy assumption \eqref{eq:V_assumptionT}. Under condition (LWP), take $(\alpha^*,\rho_1^*,\rho_2^*)$ for which (M1)--(M3) hold. Then
\[
(e^{\icomp t \omega_1^*}u_1^*,e^{\icomp t \omega_2^*}u_2^*)\quad \text{ is orbitally stable for }
\eqref{eq:NLS_with_gamma}.
\]
\end{theorem}
From now on we will work under the assumptions of Theorem \ref{thm:stability_general}.
As we mentioned in the introduction, the proof is inspired by \cite{Shatah1983}.

Let us first check the following consequence of the uniqueness property (M1).
\begin{lemma}\label{lemma:uniqueness_complexcase}
Given $\alpha,\rho_1,\rho_2>0$, we have
\begin{equation}\label{eq:Malpha_complex}
M(\alpha,\rho_1,\rho_2)=\sup\left\{ F(w_1,w_2):\ (w_1,w_2)\in \HcalC,\
(|w_1|,|w_2|)\in\Ucal(\alpha, \rho_1,\rho_2)\right\}.
\end{equation}
Moreover, if $(w_1,w_2)\in \HcalC$ achieves $M(\alpha^*,\rho_1^*,\rho_2^*)$, then
\[
w_1=e^{\icomp s_1}u_1^*,\qquad w_2=e^{\icomp s_2}u_2^*
\]
for some $s_1,s_2\in \R$.
\end{lemma}
\begin{proof}
Denote by $\widetilde M(\alpha,\rho_1,\rho_2)$ the right hand side of \eqref{eq:Malpha_complex};
clearly, $M(\alpha,\rho_1,\rho_2)\leq \widetilde M(\alpha,\rho_1,\rho_2)$. On the other hand, given
any $(w_1,w_2)\in \HcalC$ satisfying
\[
\|(w_1,w_2)\|^2_\Hcal\leq \alpha,\qquad \int_\Omega |w_i|^2\, dx=\rho_i,
\]
by the diamagnetic inequality\footnote{Take $w:\Omega\to \C$ such that $\int_\Omega |\nabla w|^2\,
dx<\infty$. Then $\int_\Omega |\nabla |w||^2 \, dx\leq \int_\Omega |\nabla w|^2\, dx$. Moreover,
equality holds if and only if the real and imaginary parts of $w$ are proportional functions.
See e.g. [Lieb-Loss, Analysis, Theorem 7.21].} it is clear that $(|w_1|,|w_2|)\in \Ucal(\alpha,
\rho_1,\rho_2)$ with $F(|w_1|,|w_2|)=F(w_1,w_2)$. Thus equality \eqref{eq:Malpha_complex} holds.

Let us now check the second statement of the lemma. Take $(w_1,w_2)\in \HcalC$ achieving
$M(\alpha^*,\rho_1^*,\rho_2^*)$. By the considerations of the previous paragraph, we have that
also $(|w_1|,|w_2|)$ achieves $M(\alpha^*,\rho_1^*,\rho_2^*)$, and in particular (cf. Lemma
\ref{lemma:maximizer_solves_sys})
\[
\int_\Omega |\nabla |w_i||^2\, dx=\int_\Omega |\nabla w_i|^2\, dx (=\alpha^*).
\]
Thus there exists $(u_1,u_2)\in \Hcal$ (real valued) and $k_i\in \R$ such that
\[
w_i=u_i+\icomp k_i u_i =(1+k_i\icomp) u_i =r_i e^{\icomp s_i}u_i
\]
for some $r_1,r_2>0$, $s_1,s_2\in \R$. By (M1), we have that $(|w_1|,|w_2|)=(r_1 |u_1|,r_2|u_2|)=(u_1^*,u_2^*)$, which ends the proof.
\end{proof}

\begin{lemma}\label{lemma:auxSTAB1}
Take $(\psi_1,\psi_2)\in \HcalC$ and assume that, for some $\bar\alpha\in(\alpha_1,\alpha_2)$, we have
\begin{equation}\label{eq:auxSTAB1}
\Ecal_{\gamma^*}(\psi_1,\psi_2)<\frac{\bar\alpha}{2}-\gamma^* M(\bar\alpha,\Qcal(\psi_1),\Qcal(\psi_2)).
\end{equation}
Then
\begin{align*}
\|(\psi_1,\psi_2)\|^2_\Hcal&<\bar \alpha \quad  \Rightarrow \quad  \|(\Psi_1(t),\Psi_2(t))\|^2_\Hcal<\bar\alpha \qquad \forall \, t\ \textrm{ in the existence interval}\\
\|(\psi_1,\psi_2)\|^2_\Hcal&>\bar \alpha \quad  \Rightarrow \quad  \|(\Psi_1(t),\Psi_2(t))\|^2_\Hcal>\bar\alpha \qquad \forall \, t\ \textrm{ in the existence interval}.
\end{align*}
\end{lemma}
\begin{proof}
Suppose, in view of a contradiction, that for some $\bar t$ we have
\[
\|(\Psi_1(\bar t),\Psi_2( \bar t))\|_\Hcal^2=\bar\alpha.
\]
Then, by assumption \eqref{eq:auxSTAB1} and the conservation of energy,
\begin{align*}
\frac{\bar \alpha}{2}-\gamma^* F(\Psi_1(\bar t),\Psi_2(\bar t))  &= \Ecal_{\gamma^*}(\Psi_1(\bar t),\Psi_2(\bar t))\\
												&= \Ecal_{\gamma^*}(\psi_1,\psi_2)< \frac{\bar \alpha}{2}-\gamma^* M(\bar\alpha,\Qcal(\psi_1),\Qcal(\psi_2)),
\end{align*}
which yields
\[
M(\bar\alpha,\Qcal(\psi_1),\Qcal(\psi_2))<F(\Psi_1(\bar t),\Psi_2(\bar t)).
\]
On the other hand, by conservation of mass, we have $(\Psi_1(\bar t),\Psi_2(\bar t))\in \Ucal(\bar\alpha,\Qcal(\psi_1),\Qcal(\psi_2))$, which provides a contradiction.
\end{proof}

\begin{lemma}\label{lemma:auxSTAB2}
The function
\begin{align*}
e(\alpha):&=\frac{\alpha}{2}-\gamma^* M(\alpha,\rho_1^*,\rho_2^*)\\
		&=\Ecal_{\gamma^*}(u_1(\alpha),u_2(\alpha))
\end{align*}
has a strict local minimum at $\alpha=\alpha^*$
\end{lemma}

\begin{proof}
\medbreak

\noindent \underline{Step 1.} Let
\[
\frac{d}{d\alpha} (u_1(\alpha),u_2(\alpha))
=:(v_1(\alpha),v_2(\alpha)).
\]
Differentiating the identities
\[
\sum_{i=1}^2 \int_\Omega (|\nabla u_i|^2+V_i(x)u_i^2)\, dx=\alpha, \qquad \int_\Omega u_i^2\, dx=\rho_i^* \ (i=1,2)
\]
with respect to $\alpha$, we obtain
\begin{equation}\label{eq:identity_differentiate_wrto_alpha}
\sum_{i=1}^2\int_\Omega (\nabla u_i\cdot \nabla v_i+V_i(x) u_i v_i)\, dx=\frac{1}{2}, \qquad \int_\Omega u_i v_i\, dx=0\  (i=1,2).
\end{equation}
Test the equation for $u_i$:
\[
-\Delta u_i+ (V_i(x) +\omega_i)u_i=\gamma(\mu_i u_i^3+\beta u_i u_j^2)
\]
by $v_i$; combining the result with \eqref{eq:identity_differentiate_wrto_alpha}, we obtain:
\begin{equation}\label{eq:identity_1/2}
\gamma\int_\Omega \left(\mu_1 u_1^3v_1+\mu_2u_2^3v_2+\beta u_1u_2^2 v_1 +\beta u_1^2 u_2 v_2 \right)\, dx=\frac{1}{2}.
\end{equation}

\medbreak

\noindent \underline{Step 2.} As
\[
e(\alpha)=\frac{\alpha}{2}-\gamma^*\int_\Omega \left( \frac{\mu_1 u_1^4}{4}+\frac{\mu_2 u_2^4}{4}+\frac{\beta u_1^2 u_2^2}{2}\right)\, dx,
\]
taking the derivative in $\alpha$ we see that, by step 1,
\begin{align*}
e'(\alpha)&=\frac{1}{2}-\gamma^* \int_\Omega \left(\mu_1 u_1^3v_1+\mu_2u_2^3v_2+\beta u_1u_2^2 v_1 +\beta u_1^2 u_2 v_2\right)\, dx\\
		&=\frac{1}{2}\left(1-\frac{\gamma^*}{\gamma}\right).
\end{align*}
As $\gamma(\alpha)$ is strictly increasing in a neighborhood of $\alpha^*$ (cf. assumption (M3)), the result follows.
\end{proof}

\begin{proof}[Proof of Theorem \ref{thm:stability_general}] \hspace{0.1cm}

\smallbreak

\noindent  a) \underline{Proof of property \eqref{eq:condition_stability1}.}
Fix a small $\eps$ so that $\alpha^*\pm \eps\in (\alpha_1,\alpha_2)$ and
\[
e(\alpha^*)<e(\alpha^*\pm \eps)
\]
(recall Lemma \ref{lemma:auxSTAB2}). Moreover,  take $\eta=\eta(\eps)$ so that
\[
e(\alpha^*)<e(\alpha^*\pm \eps)-\eta,
\]
which we can rewrite as
\[
\Ecal_{\gamma^*}(u_1^*,u_2^*)<\frac{\alpha^*\pm \eps}{2}-\gamma^* M(\alpha^* \pm \eps,\rho_1^*,\rho_2^*)-\eta.
\]
Then, for $\delta>0$ sufficiently small and $\|(\psi_1,\psi_2)-(u_1^*,u_2^*)\|^2_{\Hcal}<\delta$, we have
\[
\Ecal_{\gamma^*}(\psi_1,\psi_2)<\frac{\alpha^*\pm \eps}{2}-\gamma^* M(\alpha^* \pm \eps,\Qcal(\psi_1),\Qcal(\psi_2)),
\]
where we have used the $\Hcal$--continuity of $\Ecal_{\gamma^*}$ and $\Qcal$, as well as Lemma
\ref{lem:M_is_continuous}. Moreover, since we have (for $\delta$ small)
\[
\alpha^*-\delta<\| (\psi_1,\psi_2)\|_\Hcal^2<\alpha^*+\delta,
\]
then Lemma \ref{lemma:auxSTAB1} applied with $\bar\alpha=\alpha^*\pm\eps$ implies that
\[
\alpha^*-\eps<\| (\Psi_1(t),\Psi_2(t))\|_\Hcal^2<\alpha^*+\eps
\]
and in particular $(\Psi_1(t),\Psi_2(t))$ is defined for all $t\geq 0$ (as the existence interval in time is uniform with respect to the norm of the initial data, cf. (LWP)).

\medbreak

\noindent b) \underline{Proof of property \eqref{eq:condition_stability2}.}
If \eqref{eq:condition_stability2} does not hold, then we can find initial data $(\psi_{1n},\psi_{2n})\to (u_1^*,u_2^*)$ in $\HcalC$, a sequence $(t_n)_n$, and $\eta>0$ such that
\begin{equation}\label{eq:stability_contradiction_assumption}
\inf_{s_1,s_2\in \R} \| (\Psi_{1n}(t_n),\Psi_{2n}(t_n))-(e^{\icomp s_1}u_1^*,e^{\icomp s_2}u_2^*)\|_\Hcal\geq \eta
\end{equation}
(here, of course, $(\Psi_{1n},\Psi_{2n})$ is the solution to \eqref{eq:NLS_with_gamma}
corresponding to the initial datum $(\psi_{1n},\psi_{2n})$). By a), we can suppose without loss of generality that the sequences satisfy
\begin{equation}\label{eq:stability_cons_step1_1}
\Ecal_{\gamma^*}(\psi_{1n},\psi_{2n})<\frac{1}{2}\left(\alpha^*\pm \frac{1}{n}\right)-\gamma^* M\left(\alpha^*\pm \frac{1}{n},\rho_1^*,\rho_2^*\right)
\end{equation}
and
\begin{equation}\label{eq:stability_cons_step1_2}
\alpha^*-\frac{1}{n}<\|(\Psi_{1n}(t_n),\Psi_{2n}(t_n))\|^2_\Hcal<\alpha^*+\frac{1}{n}.
\end{equation}
Moreover, by the conservation of mass along trajectories,
\[
\int_\Omega |\Psi_{in}(t_n)|^2\, dx=\int_\Omega |\psi_{in}|^2\to \rho_i^*\qquad i=1,2.
\]
In particular, $(\Psi_{1n}(t_n),\Psi_{2n}(t_n))$ is bounded in $\Hcal$, hence up to a subsequence we have weak convergence in $\HcalC$ to $(w_1,w_2)$, strongly in $L^2\cap L^4$. The limiting configuration then satisfies
\begin{equation}\label{eq:compactneeded1}
\|(w_1,w_2)\|^2_\Hcal\leq \alpha^*,\qquad \int_\Omega w_i^2\, dx=\rho_i\ i=1,2
\end{equation}
so that $F(w_1,w_2)\leq M(\alpha^*,\rho_1^*,\rho_2^*)$. On the other hand, we have
\begin{align*}
\alpha^*-\frac{1}{n}-\gamma^* F(\Psi_{1n}(t_n),\Psi_{2n}(t_n))&<\Ecal_{\gamma^*}(\Psi_{1n}(t_n),\Psi_{2n}(t_n))
=\Ecal_{\gamma^*}(\psi_{1n},\psi_{2n})\\
&<\frac{1}{2}\left(\alpha^*-\frac{1}{n}\right)-\gamma^* M(\alpha^*-\frac{1}{n},\rho_1^*,\rho_2^*),
\end{align*}
where the first inequality is due to \eqref{eq:stability_cons_step1_2} and the second one to \eqref{eq:stability_cons_step1_1}. Hence
\[
M\left(\alpha^*-\frac{1}{n},\rho_1^*,\rho_2^*\right)<F(\Psi_{1n}(t_n),\Psi_{2n}(t_n))
\]
and (by (M2) and again by strong $L^4$ convergence)
\begin{equation}\label{eq:compactneeded2}
M(\alpha^*,\rho_1^*,\rho_2^*)\leq F(w_1,w_2).
\end{equation}
Thus $(w_1,w_2)$ achieves $M(\alpha^*,\rho_1^*,\rho_2^*)$ and $(\Psi_{1n}(t_n),\Psi_{2n}(t_n))\to (w_1,w_2)$
strongly in $\HcalC$. Finally, we obtain a contradiction by combining
\eqref{eq:stability_contradiction_assumption} with Lemma \ref{lemma:uniqueness_complexcase}.
\end{proof}

\begin{proof}[End of the proof of Theorem \ref{thm:intro2}]
In the assumptions of the theorem, let us fix any $\alpha\in(\alpha_1,\alpha_2)$, and relabel the
triplet $(\a,\rho_1,\rho_2)$ as $(\a^*,\rho_1^*,\rho_2^*)$. If $M(\a^*,\rho_1^*,\rho_2^*)$ is
achieved by a unique pair, then the proof follows from Theorem \ref{thm:stability_general}, once
one notices that $(\Psi_1,\Psi_2)$ solves \eqref{eq:NLS_with_gamma},
if and only if $(\Phi_1,\Phi_2)=\sqrt{\gamma^*}(\Psi_1,\Psi_2)$ solves \eqref{eq:intro_sys}.
Without the uniqueness assumption, one may repeat the proof with minor changes, observing that,
by Corollary \ref{coro:strong_conv}, the set of pairs $(u_1,u_2)$ achieving
$M(\a^*,\rho_1^*,\rho_2^*)$ is compact in $\Hcal$.
\end{proof}

\section{Applications}\label{sec:appl}

\subsection{The case of small masses}\label{subsec:AP}

To prove Theorem \ref{thm:intro3}, we will use the following parametric version of a
well known result due to Ambrosetti and Prodi \cite{AmbrosettiProdiBook}.
In the following, $\Ker$ and $\Rk$ denote respectively the kernel and the range of a
linear operator.
\begin{theorem}\label{thm:AmbrosettiProdi_parametric}
Let $X,Y$ be Banach spaces, $U\subset X$ an open set, $I\subset \R$ an open interval,
and $\Phi\in C^2(U\times I,Y)$.

Take $(x^*,\vtheta^*)\in U\times I$ such that:
\begin{enumerate}
\item\label{ass:critpointcurve}
    there exists a continuous curve $\bar x : I \to U$, with $\bar x(\vtheta^*)=x^*$, and
    \[
    \begin{cases}
    \Phi(x,\vtheta)=0\\
    (x,\vtheta)\in U\times I
    \end{cases}
    \qquad\iff\qquad
    x=\bar x(\vtheta), \ \vtheta\in I;
    \]
\item\label{ass:kernel}
     there exists $\vphi^*\in X$, non trivial, such that
     \[
     \Ker(\Phi_x(x^*,\vtheta^*))
     = \spann\{\vphi^*\};
     \]
\item\label{ass:rank}
     there exists a nontrivial $\Psi\in Y^*$ (independent of $\vtheta$) such that
     \[
     \Rk(\Phi_x (\bar x(\vtheta),\vtheta)) = \Ker\Psi\qquad\text{for every }\vtheta\in I;
     \]
\item\label{ass:nondeg}
     $\langle \Psi, \Phi_{xx}(x^*,\vtheta^*)[\vphi^*,\vphi^*] \rangle >0$.\end{enumerate}
Finally, let $z\in Y$ be such that $\langle \Psi, z\rangle = 1$.

Then there exist $\bar\delta>0$, $\bar\ep>0$ such that, for every $|\vtheta-\vtheta^*|
<\bar\delta$ the equation
\begin{equation}\label{eq:AP_abstractequation}
\Phi(x,\vtheta)= \ep z,\qquad x \in B_{\bar\delta}(x^*)
\end{equation}
has no solutions when $-\bar\eps\leq\eps<0$, while for each $0<\eps\leq\bar\eps$ it
has exactly two solutions
\[
x=x_+(\eps,\vtheta),\qquad x=x_-(\eps,\vtheta).
\]
Furthermore, the maps $x_\pm: (0,\bar\eps]\times(\vtheta^*-\bar\delta, \vtheta^*+\bar\delta)\to
B_{\bar\delta}(x^*)$ are of class $C^2$ and continuous up to $\eps=0^+$. More precisely,
\begin{equation}\label{eq:ambrosetti_prodi_expansion}
x_\pm(\eps,\vtheta) = \bar x(\vtheta) + \left[ \vphi^* +
\eta_\pm(\ep,\vtheta)\right]t_\pm (\eps,\vtheta),
\end{equation}
where the maps $\eta_\pm$ are $C^1([0,\bar\ep]\times(\vtheta^*-\bar\delta, \vtheta^*+\bar\delta))$
with $\eta(0,\vtheta^*)=0$, while the
functionals $t_\pm$ are $C^2$ for $\eps>0$, continuous (and vanishing) up to $\eps=0^+$, and
\[
\pm t_\pm(\eps,\vtheta) >0,
\quad \pm\frac{\partial t_\pm(\ep,\vtheta)}{\partial\ep} \geq \frac{C}{\sqrt\ep},
\qquad\text{ for }\eps\in(0,\bar\ep],
\]
for a suitable $C>0$ (related to the positive number appearing in assumption \eqref{ass:nondeg}).
\end{theorem}
The proof of such theorem follows very closely the one of the original Ambrosetti-Prodi result
\cite[Section 3.2, Lemma 2.5]{AmbrosettiProdiBook}, taking however into account the
dependence on the parameter  $\vtheta$, which is not present in the latter. For the reader's convenience,
here we summarize the proof, enlightening the main differences.
\begin{proof}
To start with, we apply a Lyapunov-Schmidt reduction to equation \eqref{eq:AP_abstractequation}.
To this aim, let $L:= \Phi_x(x^*,\vtheta^*) \in\mathcal{L}(X,Y)$, and let $W\subset X$
denote a topological complement of $\Ker L$. Then, for every $x\in X$, we can write
\[
x = x^* + t\vphi^* + w,\qquad \text{for unique }t\in\R,\,w\in W.
\]
Analogously, since $Y=\spann\{z\} \oplus \Rk L$, for every $y\in Y$ we can uniquely write
\[
y =: \underbrace{P y}_{\in\spann\{z\}} + \underbrace{Q y}_{\in \Rk L},\qquad
\text{where }P y = \langle \Psi, y \rangle z
\]
($\Psi$ appearing in assumption \eqref{ass:rank}).
Using such decompositions, equation \eqref{eq:AP_abstractequation} writes
\begin{equation}\label{eq:LyapSchm}
\begin{cases}
P \Phi(x,\vtheta)= \langle \Psi , \Phi(x^* + t\vphi^* + w,\vtheta) \rangle z = \ep z \\
Q \Phi(x,\vtheta)= Q \Phi(x^* + t\vphi^* + w,\vtheta) = 0.
\end{cases}
\end{equation}
Now, by construction, we can apply the Implicit Function Theorem to the second equation
in order to solve for $w$ near $(t,\vtheta,w)=(0,\vtheta^*,0)$
(indeed, in such point, the partial derivative of the l.h.s. with respect to $w$ is $L:W\to\Rk L$,
which is invertible). As a consequence, for some positive $\delta$,
\[
    \begin{cases}
    Q \Phi(x^* + t\vphi^* + w,\vtheta)=0\\
    (t,\vtheta-\vtheta^*,w)\in [-\delta,\delta]^2\times B'_\delta(0)
    \end{cases}
    \iff
    w=w(t,\vtheta), \ (t,\vtheta-\vtheta^*)\in [-\delta,\delta]^2
\]
(here $B'$ denotes the ball in $W$). For future reference, we notice that, possibly decreasing
$\delta$, the $C^2$ function $w$ satisfies
\begin{equation}\label{eq:trivial_w}
w(0,\vtheta)=\bar x(\vtheta) - x^*\qquad\text{for every } \vtheta-\vtheta^*\in [-\delta,\delta].
\end{equation}
Indeed, since $\bar x(\vtheta^*)=x^*$, this follows from the fact that $w$ is the unique solution of the above equation near
$(t,\vtheta,w)=(0,\vtheta^*,0)$, together with the fact that
\[
Q \Phi(x^* + 0\cdot\vphi^* + (\bar x(\vtheta) - x^*),\vtheta)  =
Q \Phi(\bar x(\vtheta),\vtheta)=0
\]
by assumption \eqref{ass:critpointcurve}. Furthermore, we also have that
\begin{equation}\label{eq:trivial_w_t}
w_t(0,\vtheta^*)=0.
\end{equation}
Indeed, taking the partial derivative
of the second equation in \eqref{eq:LyapSchm} with respect to $t$, we obtain
\[
Q \Phi_x(x^* + t\vphi^* + w(t,\vtheta),\vtheta)[\vphi^*+w_t(t,\vtheta)] = 0;
\]
for $(t,\vtheta)=(0,\vtheta^*)$, this yields
\[
0 =  Q \Phi_x(x^* ,\vtheta^*)[\vphi^*+w_t(0,\vtheta^*)] =  L w_t(0,\vtheta^*),
\]
so that $w_t(0,\vtheta^*)\in \Ker L$. Since $w_t(0,\vtheta^*)\in W$ by definition, \eqref{eq:trivial_w_t} follows.

Substituting $w=w(t,\vtheta)$ in the first equation in \eqref{eq:LyapSchm} we obtain the bifurcation
equation
\begin{equation}\label{eq:chi=eps}
\text{find }(t,\vtheta-\vtheta^*)\in [-\delta,\delta]^2\text{ s.t. }\qquad
\chi(t,\vtheta):= \langle \Psi , \Phi(x^* + t\vphi^* + w(t,\vtheta),\vtheta) \rangle  = \ep,
\end{equation}
which is locally equivalent to \eqref{eq:AP_abstractequation}. Equation \eqref{eq:trivial_w} implies
that, for every $\vtheta-\vtheta^*\in [-\delta,\delta]$,
\[
\chi(0,\vtheta) =\langle \Psi , \Phi(\bar x(\vtheta),\vtheta) \rangle= 0.
\]
On the other hand, direct calculations yield
\[
\begin{split}
\chi_t(t,\vtheta) &= \langle \Psi , \Phi_x(x^* + t\vphi^* + w(t,\vtheta),
\vtheta)[\vphi^*+w_t(t,\vtheta)] \rangle\\
\chi_{tt}(t,\vtheta) &= \langle \Psi , \Phi_{xx}(x^* + t\vphi^* + w(t,\vtheta),
\vtheta)[\vphi^*+w_t(t,\vtheta)]^2\\
& \qquad\qquad\qquad\qquad +  \Phi_x(x^* + t\vphi^* + w(t,\vtheta),
\vtheta)[w_{tt}(t,\vtheta)]\rangle.
\end{split}
\]
Using assumption \eqref{ass:rank} and equations \eqref{eq:trivial_w}, \eqref{eq:trivial_w_t}, we infer
\begin{equation}\label{eq:chitt}
\begin{split}
\chi_t(0,\vtheta) &= 0\\
\chi_{tt}(0,\vtheta^*) &= \langle \Psi , \Phi_{xx}(x^*,\vtheta^*)[\vphi^*,\vphi^*]\rangle>0
\end{split}
\end{equation}
by assumption \eqref{ass:nondeg}. Since $\chi$ is $C^2$, we can
find positive constants $C_1$, $C_2$ such that
\[
\begin{cases}
2C_1 \leq \chi_{tt}(t,\vtheta) \leq 2C_2\\
2C_1|t| \leq \mathrm{sign}(t) \chi_{t}(t,\vtheta) \leq 2C_2|t|\\
C_1 t^2 \leq \chi(t,\vtheta) \leq C_2 t^2
\end{cases}
\qquad \text{for every }(t,\vtheta - \vtheta^*)\in[-\bar\delta,\bar\delta]^2,
\]
for some suitable $\bar\delta\leq\delta$.
As a first consequence, \eqref{eq:chi=eps} is not solvable for $\eps<0$. Furthermore, defining
\[
\bar\eps := \min_{|\vtheta - \vtheta^*|\leq \bar\delta} \chi(\pm\bar\delta,\vtheta) > 0
\]
we deduce that, for every $\vtheta - \vtheta^*\in[-\bar\delta,\bar\delta]$ and $\eps\in(0,\bar\eps]$, there
exist $-\bar\delta \leq t_-(\eps,\vtheta) < 0 < t_+(\eps,\vtheta)\leq\bar\delta$ such that
\[
    \begin{cases}
    \chi(t,\vtheta)=\ep\\
    (t,\vtheta-\vtheta^*,\ep)\in [-\bar\delta,\bar\delta]^2\times (0,\bar\eps]
    \end{cases}
    \iff
    t=t_\pm(\eps,\vtheta), \ (\vtheta-\vtheta^*,\ep)\in [-\bar\delta,\bar\delta]\times (0,\bar\eps].
\]
Clearly $t_\pm(0^+,\vtheta)=0$, uniformly in $\vtheta$. Moreover, since $\chi_t(t,\vtheta)\neq 0$
for $t\neq0$, the Implicit Function Theorem implies that the maps $t_\pm$ are $C^2$ for $\ep>0$,
with
\[
\pm\frac{\partial t_\pm(\ep,\vtheta)}{\partial\ep} = \frac{\pm 1}{\chi_t(t_\pm(\ep,\vtheta),\vtheta)}
\geq \frac{1}{2C_2|t_\pm(\ep,\vtheta)|} \geq \frac{1}{2C_2} \sqrt{\frac{C_1}{\chi(t_\pm(\ep,\vtheta),
\vtheta)}} = \frac{\sqrt{C_1}}{2C_2} \frac{1}{\sqrt\ep}.
\]
Setting
\[
\begin{split}
x_\pm(\eps,\vtheta) &= x^* + t_\pm (\eps,\vtheta) \vphi^* + w(t_\pm (\eps,\vtheta),\vtheta)\\
&= \bar x(\vtheta) + t_\pm (\eps,\vtheta) \vphi^* + w(t_\pm (\eps,\vtheta),\vtheta)-
(\bar x(\vtheta)-x^*)\\
&= \bar x(\vtheta) + \left[ \vphi^* + \frac{w(t_\pm (\eps,\vtheta),\vtheta) -
w(0,\vtheta)}{t_\pm (\eps,\vtheta)} \right]t_\pm (\eps,\vtheta) ,
\end{split}
\]
one can complete the proof by recalling that the maps
\[
\eta_\pm(\ep,\vtheta) :=
\begin{cases}
[w(t_\pm (\eps,\vtheta),\vtheta) -  w(0,\vtheta)]/t_\pm (\eps,\vtheta)  &  \eps\neq 0\\
w_t(0,\vtheta)  & \eps = 0
\end{cases}
\]
are $C^1$ up to $\eps=0$, and that $\eta(0,\vtheta^*)=0$ by equation \eqref{eq:trivial_w_t}.
\end{proof}
\begin{remark}
The following uniform in $\vtheta$ limit:
\begin{multline*}
\lim_{\eps\to0^+} \pm \frac{t_\pm (\eps,\vtheta)}{\sqrt{\eps}} =
\lim_{t \to0^+} \frac{t}{\sqrt{\chi(\pm t,\vtheta)}} =
\lim_{t \to0^+} \frac{t}{\sqrt{\chi_{tt}(\xi_\pm,\vtheta) t^2/2}} \\
\text{(for suitable $-t<\xi_-<0< \xi_+ < t$) }\qquad
=  \sqrt{\frac{2}{\chi_{tt}(0,\vtheta)}} =: a(\vtheta),
\end{multline*}
\[
\lim_{\eps\to0^+} \frac{w(t_\pm (\eps,\vtheta),\vtheta) - w(0,\vtheta)}{t_\pm (\eps,\vtheta)}
= w_t( 0, \vtheta) = \eta(0,\vtheta),
\]
implies that, as $\eps\to0^+$,
\begin{equation*}
x_\pm(\eps,\vtheta) = \bar x(\vtheta) \pm a(\vtheta)\left[\vphi^* + \eta (0,\vtheta)\right]\sqrt{\eps}
+ o(\sqrt{\eps}),\qquad
\text{ uniformly in }\vtheta,
\end{equation*}
where
\[
a(\vtheta^*) = \sqrt{\frac{2}{\langle \Psi, \Phi_{xx}(x^*,\vtheta^*)[\vphi^*,\vphi^*]\rangle}}>0.
\]
\end{remark}
\begin{remark}
A point $(x^*,\vtheta^*)$ satisfying assumptions \eqref{ass:kernel}, \eqref{ass:rank} and
\eqref{ass:nondeg}  in Theorem \ref{thm:AmbrosettiProdi_parametric} (the latter ones for $\vtheta=
\vtheta^*$) is said to be \emph{ordinary singular} for $\Phi$. As a matter of fact, assumption
\eqref{ass:critpointcurve} insures not only that $(x^*,\vtheta^*)$ is ordinary singular for $\Phi$,
but also that $(\bar x(\vtheta),\vtheta)$ exhibits an ordinary singular type geometry, at least for
$|\vtheta-\vtheta^*|$ small.
\end{remark}

In order to apply the previous abstract result, let $X= \Hcal \times \R^3$, $Y= \Hcal^*  \times \R^3$ and take the $C^2$ map $\Phi:X\times  \R \to Y$ defined by
\begin{equation}\label{eq:definition_of_Phi}
\Phi(u_1,u_2,\omega_1,\omega_2,\gamma,\vtheta)=
\left(
\begin{array}{c}
\Delta u_1-(V_1(x)+\omega_1) u_1+\gamma(\mu_1 u_1^3+\beta u_1 u_2^2)\\
\Delta u_2-(V_2(x)+\omega_2) u_2+\gamma(\mu_2 u_2^3+\beta u_1^2 u_2) \smallskip \\
 \int_\Omega u_1^2\, dx-(\cos^2 \vtheta)/\lambda_{V_1}\smallskip \\
 \int_\Omega u_2^2\, dx-(\sin^2 \vtheta)/\lambda_{V_2} \smallskip \\
 \sum_{i=1}^2\int_\Omega \left(|\nabla u_i|^2+V_i(x)u_i^2\right)\, dx-1
\end{array}
\right).
\end{equation}
\begin{remark}\label{rem:M_iff_Phi}
Note that, recalling the definition of $\tilde{\mathcal{U}}(\a,\rho_1^*,\rho_2^*)$
from Section \ref{sec:varpro}, we have that $\Phi(u_1,u_2,\omega_1,\omega_2,\gamma,\vtheta)=(0,0,0,0,\eps)$ if and only
if
\[
(u_1,u_2) \in \tilde{\mathcal{U}}\left(1+\ep,\frac{\cos^2 \vtheta}{\lambda_{V_1}},
\frac{\sin^2 \vtheta}{\lambda_{V_2}}\right)
\text{ and equation \eqref{eq:system_main} holds.}
\]
\end{remark}
Finally, we define $\bar x:\R\to X$ as
\begin{equation}
\begin{split}
\bar x(\vtheta)&:=(\bar u_1(\vtheta),\bar u_2(\vtheta),\bar \omega_1(\vtheta),\bar \omega_2(\vtheta),\bar \gamma(\vtheta))\\&=\left(\frac{\cos \vtheta}{\sqrt{\lambda_{V_1}}}\vphi_{V_1}  , \frac{\sin \vtheta}{\sqrt{\lambda_{V_2}}} \vphi_{V_2},-\lambda_{V_1},-\lambda_{V_2},0\right),\text{ and}\\
(\bar\rho_1(\vtheta),\bar\rho_2(\vtheta)) &:=
\left(\frac{\cos^2 \vtheta}{\lambda_{V_1}},\frac{\sin^2 \vtheta}{\lambda_{V_2}}\right).
\end{split}
\end{equation}
In the following, we will systematically adopt the above notation, possibly
dropping the explicit dependence on $\vtheta$ when no confusion may arise.

We start with the following lemma, which will ensure that assumption \eqref{ass:critpointcurve}
in Theorem \ref{thm:AmbrosettiProdi_parametric} holds for
$\bar x(\cdot)$ in $I=(0,\pi/2)$ (and suitable $U$).
\begin{lemma}\label{lem:ass1AP}
Take $\vtheta\not\in \Z\pi/2$ and $\eps_n\to 0^+$, and suppose that
\[
\Phi(u_{1,n},u_{2,n}, \omega_{1,n}, \omega_{2,n},\gamma_n, \vtheta)=(0,0,0,0,\eps_n).
\]
Then, up to subsequences, $\omega_{i,n}\to \bar\omega_{i}(\vtheta)$ ($i=1,2$), $\gamma_n\to 0$, and
\[
u_{1,n}\to (-1)^l\bar u_{1}( \vtheta),\quad
u_{2,n}\to (-1)^m\bar u_{2}( \vtheta),\qquad
\text{strongly in }\Hcal,
\]
for some $(l,m)\in\{0,1\}^2$.

In particular, for $\vtheta\in (0,\pi/2)$, and $\tilde U\subset \Hcal$ open, containing
the above possible limits only for $l=m=0$, we have that
\begin{equation*}
\begin{array}{c}
\Phi(u_1,u_2,\omega_1,\omega_2,\gamma,\vtheta)=(0,0,0,0,0),\qquad (u_{1},u_{2})\in\tilde U\\
 \Updownarrow\\
(u_1,u_2,\omega_1,\omega_2,\gamma)=\bar x(\vtheta).
\end{array}
\end{equation*}
\end{lemma}
\begin{proof}
As $\eps_n$ is a bounded sequence, then $(u_{1,n},u_{2,n})$ is bounded in $\Hcal$ and, up to a
subsequence, $(u_{1,n},u_{2,n})\rightharpoonup (u_1,u_2)$ weakly in $\Hcal$, with $u_i$ being
nontrivial functions satisfying $\int_\Omega u_i^2\, dx=\bar\rho_i(\vtheta)$, $i=1,2$.
By definition of $\lambda_{V_i}$, we have
\begin{align*}
1\leq \sum_{i=1}^2\int_\Omega (|\nabla u_i|^2+V_i(x) u_i^2)\, dx  &\leq \liminf_n \sum_{i=1}^2\int_\Omega (|\nabla u_{i,n}|^2+V_i(x) u_{i,n}^2)\, dx\\
												&\leq \liminf_n (1+\eps_n)=1.
\end{align*}
Thus the convergence is strong, and $u_1,u_2$ are normalized eigenfunctions.
On the other hand, from Lemma \ref{lemma:bounded_quantities} we have that
$\omega_{i,n}\to \omega_i$, $\gamma_n\to \gamma$ for some constants $\omega_i,\gamma$. These must
satisfy (recall that $u_1,u_2$ are nontrivial)
\[
\lambda_{V_1}+\omega_1 = \gamma(\mu_1 u_1^2+\beta u_2^2),\qquad \lambda_{V_2}+\omega_2=\gamma(\mu_2u_2^2+\beta  u_1^2).
\]
As $u_1=u_2=0$ on $\partial \Omega$ if $\Omega$ is bounded, or $u_1,u_2$ satisfy \eqref{eq:decay_u} in case $\Omega=\R^ N$, we deduce that $\omega_i=-\lambda_{V_i}=\bar{\omega}_i$, $i=1,2$. In turn, by assumption
\eqref{eq:assumptions_nondegeneracyIntro}, $\gamma=0$.
\end{proof}
\begin{remark}
The lemma above is false for $\vtheta \in \Z\pi/2$. Indeed, for instance,
\[
\Phi\left(\frac{1}{\sqrt{\lambda_{V_1}}}\vphi_{V_1}  , 0 , -\lambda_{V_1}, \omega_2 ,0\right)
= (0,0,0,0,0)
\]
for every $\omega_2$ (and not only for $\omega_2 = -\lambda_{V_2}$).
\end{remark}
A direct computation shows that the partial derivative of $\Phi$ with respect
to the variables $x:=(u_1,u_2,\omega_1,\omega_2,\gamma)$,
\[
\Phi_{x}(x,\vtheta) :\Hcal \times \R^3\to \Hcal^*  \times \R^3,
\]
computed at $h=(v_1,v_2,o_1,o_2,g)$, yields, for $i=1,2$ and $j\neq i$:
\[
\begin{split}
(\Phi_{x}(x,\vtheta)[h])_i &= \Delta v_i-(V_i(x)+\omega_i)v_i-o_iu_i+g(\mu_i u_i^3+\beta u_iu_j^2)\\
&\qquad\qquad +\gamma(3\mu_i u_i^2 v_i+\beta u_j^2 v_i+2\beta u_1u_2 v_j),\\
(\Phi_{x}(x,\vtheta)[h])_{i+2} &=2\int_\Omega u_i v_i\, dx\\
(\Phi_{x}(x,\vtheta)[h])_{5} &= 2\sum_{i=1}^2 \int_\Omega(\nabla u_i\cdot \nabla v_i+ V_i(x) u_i v_i )\, dx.
\end{split}
\]
\begin{lemma}\label{lem:AP_ass_are_satisfied}
Given $\vtheta\in (0,\pi/2)$, denote
\[
\begin{split}
L_\vtheta &:= \Phi_{x}(\bar x(\vtheta),\vtheta).
\end{split}
\]
Then:

a) $\Ker L_\vtheta$ has dimension one, being spanned by the vector
\[
\vphi^*:=(\psi_1(\vtheta),\psi_2(\vtheta),o_1(\vtheta), o_2(\vtheta),1),
\]
where
\[
o_i(\vtheta) = \frac{1}{\bar\rho_i}
\int_\Omega (\mu_i \bar u_i^2 + \b \bar u_j^2)\bar u_i^2\, dx,
\]
and $\psi_i(\vtheta)$ is the unique solution of
\[
-\Delta\psi_i+V_i(x)\psi_i-\lambda_{V_i}\psi_i = \mu_i \bar u_i^3 + \b  \bar u_i \bar u_j^2
 - o_i \bar u_i \quad
\text{ with }\int_\Omega \psi_i \bar u_i\, dx=0;
\]

b) $\Rk L_\vtheta=\Ker \Psi$, where $\Psi:\Hcal^* \times \Hcal^* \times \R^3\to \R$ is defined by
\[
\Psi(\xi_1,\xi_2,h_1,h_2,k)=k-(\lambda_{V_1} h_1+\lambda_{V_2}h_2);
\]

c) $\Psi \left(\Phi_{xx}(\bar x(\vtheta),\vtheta)[\vphi^*,\vphi^*]\right)>0$.
\end{lemma}

\begin{proof}
%
%
\emph{a)} Let for the moment $\zeta_i(\vtheta) = \mu_i \bar u_i^3(\vtheta) + \b  \bar u_i(\vtheta) \bar u_j^2(\vtheta)$, $i=1,2,\ j\neq i$.
The equation $L_\vtheta[v_1,v_2,o_1,o_2,g]=(0,0,0,0,0)$ is equivalent to ($i=1,2$,
$j\neq i$):
\[
\begin{split}
 &-\Delta v_i+V_i(x)v_i -\lambda_{V_i}v_i  =-o_i \bar{u}_i +g\zeta_i,\\
& \int_\Omega \bar u_i v_i\, dx= \int_\Omega (\nabla \bar u_i\cdot \nabla v_i+V_i(x)\bar u_i v_i)\, dx=0.
\end{split}
\]
Testing the $i$--th equation by $\bar u_i$, one obtains each $o_i$ in function of $g$:
\[
o_i=\frac{g}{\bar \rho_i}\int_\Omega \zeta_i \bar u_i\, dx = g o_i(\vtheta).
\]
Therefore one has to solve
\[
-\Delta\psi_i +V_i(x)\psi_i-\lambda_{V_i}\psi_i = g \left[\zeta_i  -
\frac{\bar u_i}{\bar\rho_i}\int_\Omega \zeta_i \bar u_i\, dx\right]\quad
\text{ with }\int_\Omega \psi_i \bar u_i\, dx=0,
\]
and this can be uniquely done by choosing $v_i=g\psi_i(\vtheta)$, by Fredholm's Alternative.

\medbreak

\emph{b)} Take $(\xi_1,\xi_2,h_1,h_2,k)=L_\vtheta(v_1,v_2,o_1,o_2,g)\in \Rk L_\vtheta$. Then from the last three equations of this identity, and the fact that $\bar{u}_i$ are eigenfunctions, we deduce that
\begin{align*}
k&=2\sum_{i=1}^2 \int_\Omega(\nabla \bar{u}_i \cdot \nabla v_i + V_i(x)\bar{u}_i v_i)\, dx\\
  &=2\sum_{i=1}^2 \lambda_{V_i} \int_\Omega \bar{u}_i v_i \, dx =\sum_{i=1}^2 \lambda_{V_i} h_i.
\end{align*}
which shows that $\Rk L_\vtheta\subset \Ker \Psi$. Reciprocally, given $(\xi_1,\xi_2,h_1,h_2,k)\in
\Hcal^*\times\R^3$ with $k=\lambda_{V_1}h_1+\lambda_{V_2}h_2$, let $w_i$ (for $i=1,2$) be
the solution to
\[
-\Delta w_i +V_i(x)w_i-\lambda_{V_i}w_i=\vphi_{V_i}\langle \xi_i, \vphi_{V_i} \rangle - \xi_i,\qquad \int_\Omega w_i \vphi_{V_i}\, dx=0,
\]
which exists, unique, by Fredholm's Alternative. Then
\begin{multline*}
L_\vtheta\left[\frac{h_1}{2\sqrt{\bar\rho_1}}\vphi_{V_1}+w_1,\frac{h_2}{2\sqrt{\bar\rho_2}}
\vphi_{V_2}+w_2,\frac{1}{\sqrt{\bar\rho_1}}\langle \xi_1, \vphi_{V_1} \rangle,
\frac{1}{\sqrt{\bar\rho_2}}\langle \xi_2, \vphi_{V_2} \rangle,0\right]\\
=(\xi_1,\xi_2,h_1,h_2,k).
\end{multline*}

\medbreak

\emph{c)} One can check directly that
\[
\Phi_{xx}(\bar x(\vtheta),\vtheta)[\psi_1(\vtheta),\psi_2(\vtheta),o_1(\vtheta), o_2(\vtheta),1]^2
\]
is given by
\[
\left(\begin{array}{c}
-2o_1 \psi_1 +2 (3\mu_1 \bar u_{1}^2 \psi_1+\beta \bar u_{2}^2 \psi_1+2\beta
\bar u_{1}\bar u_{2}\psi_2) \\
-2o_2 \psi_2 +2 (3\mu_2 \bar u_{2}^2 \psi_2+\beta \bar u_{1}^2 \psi_2+2\beta
\bar u_{1}\bar u_{2}\psi_1)\\
2\int_\Omega \psi_1^2\, dx\\
2\int_\Omega \psi_2^2\, dx\\
2\sum_{i=1}^2 \int_\Omega (|\nabla \psi_i|^2+V_i(x)\psi_i^2)\, dx
\end{array}
\right).
\]
Its image through $\Psi$ is
\[
2\sum_{i=1}^2 \left(\int_\Omega (|\nabla \psi_i|^2+V_i(x)\psi_i^2)\, dx- \lambda_{V_i} \int_\Omega \psi_i^2\, dx\right),
\]
which is strictly positive since $\psi_i \not\equiv 0$ and $\int_\Omega \psi_i \vphi_{V_i}\, dx=0$.
\end{proof}

Now we are in position to apply Theorem \ref{thm:AmbrosettiProdi_parametric}.

\begin{lemma}\label{lem:small_mass}
For every $\vtheta^*\in(0,\pi/2)$ there exist $\bar\delta,\bar\varepsilon$ such that for every
$\vtheta\in(\vtheta^*-\bar\delta,\vtheta^*+\bar\delta)$ the problem
\begin{equation*}
\begin{cases}
\Phi(u_1,u_2,\omega_1,\omega_2,\gamma,\vtheta)=(0,0,0,0,\eps),\\
(u_1,u_2,\omega_1,\omega_2,\gamma)\in X
\end{cases}
\end{equation*}
has exactly two positive solutions $x_\pm=x_\pm(\eps,\vtheta)$ for each $0<\eps\leq\bar\eps$,
and no solution for $\eps<0$.  Moreover, $\gamma_-<0<\gamma_+$,
\[
(u_{1+}(\eps,\vtheta),u_{2+}(\eps,\vtheta))\text{ achieves } M \left(1+\ep,\bar\rho_1(\vtheta),
\bar\rho_1(\vtheta)\right)
\]
(uniquely among positive solutions) and
\[
\frac{\partial\gamma_+(\eps,\vtheta)}{\partial\eps} \geq C>0 \qquad
\text{for every }(\eps,\vtheta)\in(0,\bar\ep]\times(\vtheta^*-\bar\delta,\vtheta^*+\bar\delta).
\]
\end{lemma}
\begin{proof}
In view of Lemmas \ref{lem:ass1AP} and \ref{lem:AP_ass_are_satisfied}, most part of the
statement is a direct consequence of Theorem \ref{thm:AmbrosettiProdi_parametric}. Indeed,
under the above notation, by choosing $z=(0,0,0,0,1)$, we have $\Psi(z)=1$ and thus the
existence of $\bar\delta,\bar\varepsilon$ and $x_\pm$. One can use again Lemma \ref{lem:ass1AP}
to insure that $x_\pm$ are the only two solutions not only locally, but also among all positive
solutions. Now, the last component of equation \eqref{eq:ambrosetti_prodi_expansion} writes
\[
\gamma_\pm(\eps,\vtheta) = \left[ 1 + \tilde\eta_\pm(\ep,\vtheta)\right]t_\pm (\eps,\vtheta),
\]
where the functions $t_\pm$ satisfy
\[
\pm t_\pm(\eps,\vtheta) >0, \quad  t_\pm(0^+,\vtheta) =0,
\quad \pm\frac{\partial t_\pm(\ep,\vtheta)}{\partial\ep} \geq \frac{C}{\sqrt\ep},
\qquad\text{ for }\eps\in(0,\bar\ep),
\]
while the functions $\tilde\eta_\pm$ are $C^1$ up to $\ep=0$, with $\tilde\eta_\pm(0,\vtheta^*)=0$.
In particular, by taking possibly smaller values of $\bar\delta,\bar\ep$, we can assume that
\[
|\tilde\eta_\pm(\ep,\vtheta)|\leq\frac12,
\qquad |\partial_\ep\tilde\eta_+(\ep,\vtheta)t_+(\ep,\vtheta)|\geq\frac{C}{3\sqrt{\bar\ep}}.
\]
This is sufficient to insure that $\gamma_+>0$ and $\gamma_-<0$ so that only $(u_{1+},u_{2+})$ achieves $M$ (Lemmas \ref{lemma:maximizer_solves_sys}, \ref{lemma:gamma>0} and Remark
\ref{rem:M_iff_Phi}). Furthermore, this also implies that
\[
\frac{\partial\gamma_+}{\partial\eps} = [1+\tilde\eta_+]\frac{\partial t_+}{\partial\eps}
+ \frac{\partial \tilde\eta_+}{\partial\eps}t_+ \geq \frac12 \frac{C}{\sqrt\ep} -
\frac13 \frac{C}{\sqrt{\bar\ep}}>0
\qquad\text{ for }\eps\in(0,\bar\ep). \qedhere
\]
\end{proof}
\begin{proof}[Proof of Theorem \ref{thm:intro3}]
As usual, recall that pairs $(u_1,u_2)$ achieving $M(\a,\rho_1,\rho_2)$ correspond to
pairs $(U_1,U_2)=\sqrt{\gamma}(u_1,u_2)$ which solve system \eqref{eq:intro_sys_solitary}
with $m_i=\gamma\rho_i$.
Let $k\geq1$ be fixed. Choosing $\rho_i=\bar\rho_i(\vtheta)$, we have that
\[
\frac{1}{k}\leq \frac{m_2}{m_1} = \frac{\rho_2}{\rho_1}\leq{k}
\quad\iff\quad
\vtheta_-:=\arctan{\sqrt{\frac{\lambda_{V_2}}{k\lambda_{V_1}}}}\leq\vtheta\leq
\arctan{\sqrt{\frac{k\lambda_{V_2}}{\lambda_{V_1}}}}=:\vtheta_+.
\]
Now, since $[\vtheta_-,\vtheta_+]\subset(0,\pi/2)$, for every $\vtheta^*\in
[\vtheta_-,\vtheta_+]$ we can apply Lemma \ref{lem:small_mass}. By
compactness, we end up with a uniform $\bar\ep>0$ such that, writing $\a=1+\ep$ and
\[
x_+(\a-1,\vtheta) = (u_1(\a,\vtheta),u_2(\a,\vtheta),\omega_1(\a,\vtheta),
\omega_2(\a,\vtheta),\gamma(\a,\vtheta)),
\]
we have that $(u_1(\a,\vtheta),u_2(\a,\vtheta))$ achieves $M(\a,\bar\rho_1(\vtheta),\bar\rho_2(\vtheta))$,
uniquely among positive pairs, for every $\a\in(1,1+\bar\ep]$,
$\vtheta\in[\vtheta_-,\vtheta_+]$. As a consequence, Theorem \ref{thm:intro1}
provides the existence of the corresponding solitary waves. Furthermore,
$x^+$ is $C^1$ and
\[
\frac{\partial \gamma_+(\a,\vtheta)}{\partial \a}>0,
\qquad
\text{for every }\a\in(1,1+\bar\ep],\ \vtheta\in[\vtheta_-,\vtheta_+].
\]
Applying Theorem \ref{thm:intro2}, for $\vtheta$ fixed, by uniqueness we obtain that
the solitary waves are stable. Recalling that $\gamma(1^+,\vtheta)\equiv0$,
the theorem follows by choosing
\[
\bar m := \min_{\vtheta\in[\vtheta_-,\vtheta_+]}
\left[\bar\rho_1(\vtheta)+\bar\rho_2(\vtheta)\right]
\sqrt{\gamma(\bar\ep,\vtheta)}>0.
\qedhere
\]
\end{proof}

\subsection{Defocusing system with weak interaction}\label{sec:defocusing}

The purpose of this section is to prove the following.

\begin{theorem}\label{thm:defocusing_weak_interaction}
Let $\m_1,\mu_2<0$ and $\b^2<\m_1\m_2$. Let $V_1,V_2$ satisfy \eqref{eq:V_assumptionT}.

For every $\rho_1,\rho_2>0$ the set
\[
\mathcal{S}=\left\{ (u_1,u_2,\omega_1,\omega_2,\gamma,\alpha) \in \mathcal{H}\times\R^4: \, \begin{array}{c}
\gamma>0,\ \alpha>\lambda_{V_1}\rho_1+\lambda_{V_2}\rho_2, \\
\|(u_1,u_2)\|_{\mathcal{H}}^2 =\a, \ \int_\Omega u_i^2\,dx=\rho_i, \\
u_i>0, \  \text{ system \eqref{eq:system_main} holds}
\end{array} \right\}
\]
is a smooth curve which can be parameterized by a unique map
\[
\alpha \mapsto (u_1(\alpha),u_2(\alpha),\omega_1(\alpha),\omega_2(\alpha),\gamma(\alpha)),
\]
so that $(u_1(\a),u_2(\a))$ achieves $M(\a,\rho_1,\rho_2)$.
Moreover, if $(u_1,u_2,\omega_1,\omega_2,\gamma,\cdot)\in\mathcal{S}$, then $\gamma$ is increasing
to $+\infty$. As a consequence,
the standing wave $(e^{\icomp t \omega_1}\sqrt{\gamma}u_1, e^{\icomp t \omega_2}\sqrt{\gamma}u_2)$
is the only positive solution to \eqref{eq:intro_sys_solitary} corresponding to $m_i=\gamma\rho_i$.
Furthermore, it is conditionally orbitally stable for \eqref{eq:intro_sys}, in the sense of
Theorem \ref{thm:intro2}.
\end{theorem}

In the rest of the section, we assume that $\m_1,\mu_2<0$, $\b^2<\m_1\m_2$, and $V_1,V_2$
satisfy \eqref{eq:V_assumptionT}. In particular, \eqref{eq:M_alpha} rewrites as
\[
-M(\a,\rho_1,\rho_2)=\inf_{\tilde{\mathcal{U}}(\a,\rho_1,\rho_2)}
\int_\Om \left( \frac{|\m_1|}{4}u_1^4-\frac{\b}{2}u_1^2u_2^2 +\frac{|\m_2|}{4}u_2^4 \right)\,dx.
\]
Since $\beta^2<\mu_1\mu_2$, we are minimizing a functional which is positive and coercive. Moreover, a convexity property holds in the following sense:
\[
 |m_1|u_1^4-2 b u_1^2u_2^2 + |m_2| u_2^4 = G(u_1^2,u_2^2),\quad \text{with $G$ convex.}
\]
This is sufficient to ensure the following uniqueness result.

\begin{lemma}\label{lem:uniqueness_mu_i<0}
For fixed $\omega_1,\omega_2\in\R$ and $\gamma>0$ there is at most one positive solution
$(u_1,u_2)\in\mathcal{H}$ of system \eqref{eq:system_main}.
\end{lemma}
\begin{proof}
In the case $\Omega$ bounded, the result is proved in \cite[Theorem 4.1]{AlamaBronsardMironescu2009}.
If $\Omega=\R^N$, $\beta<0$ and $V_1(x)=V_2(x)=|x|^2$, the uniqueness is shown in
\cite{AftalionNorisSourdis2014}. Let us give a
sketch of such proof. Take two couples of solutions $(u_1,u_2)$ and $(v_1,v_2)$ and let $w_i=u_i/v_i$.
Then
\[
-\nabla\cdot(v_i^2\nabla w_i)=u_i\Delta v_i-v_i\Delta u_i
=\gamma w_iv_i^2\left( \m_iv_i^2(w_i^2-1)+\b v_j^2(w_j^2-1) \right).
\]
We test by $(w_i^2-1)/w_i$ in a ball of radius $R$ to obtain
\[
\begin{split}
\int_{B_R} v_i^2 |\nabla w_i|^2\left( 1+\frac{1}{w_i^2} \right)\,dx
=\gamma \int_{B_R} v_i^2 (w_i^2-1) \left( \m_iv_i^2(w_i^2-1)+\b v_j^2(w_j^2-1) \right) \\
+\int_{\partial B_R} \left[ \left(u_i-\frac{v_i^2}{u_i}\right)\nabla u_i-\left(\frac{u_i^2}{v_i}-v_i\right)\nabla v_i \right]\cdot\nu \,d\sigma.
\end{split}
\]
Since $\m_i<0$ and $\b^2<\m_1\m_2$, there exists $\kappa>0$ such that
\begin{equation}\label{eq:kappa_def}
|\b|\leq \sqrt{|\m_1|-\kappa}\sqrt{|\m_2|-\kappa},
\end{equation}
so that the previous equality implies
\[
\begin{split}
\sum_{i=1}^2 \int_{B_R} \left[ v_i^2 |\nabla w_i|^2\left( 1+\frac{1}{w_i^2} \right)
+\gamma kv_i^4(w_i^2-1)^2 \right] \,dx\\
\leq \sum_{i=1}^2 \int_{\partial B_R} \left[ \left(u_i-\frac{v_i^2}{u_i}\right)\nabla u_i-\left(\frac{u_i^2}{v_i}-v_i\right)\nabla v_i \right]\cdot\nu \,d\sigma.\\
\end{split}
\]
In \cite[Proposition 2.3]{AftalionNorisSourdis2014}, suitable a priori
estimates are obtained, which yield the existence of a sequence $R_k\to\infty$ such that
\begin{equation}\label{eq:boundary_term_R_k}
\sum_{i=1}^2 \int_{\partial B_{R_k}} \left[ \left(u_i-\frac{v_i^2}{u_i}\right)\nabla u_i-\left(\frac{u_i^2}{v_i}-v_i\right)\nabla v_i \right]\cdot\nu \,d\sigma \to 0
\end{equation}
as $R_k\to\infty$, which provides $u_i=v_i$ for $i=1,2$.

The same scheme can be applied also in the case of more general potentials satisfying
\eqref{eq:V_assumptionT}, exploiting the following a priori estimates.
\end{proof}

\begin{lemma}
Let $V_1,V_2$ satisfy \eqref{eq:V_assumptionT}, $\mu_1,\mu_2<0$, and $\beta<\sqrt{\mu_1\mu_2}$. There exist constants $C,c_0,R_0>0$, depending only on $\mu_1,\mu_2,\beta,V_1,V_2,\omega_1,\omega_2$, such that every positive solution $(u_1,u_2)\in\mathcal{H}$ of
\[
\left\{\begin{array}{ll}
-\Delta u_1+(V_1(x)+\omega_1)u_1=\mu_1u_1^3+\beta u_1u_2^2 \quad & \R^N \\
-\Delta u_2+(V_2(x)+\omega_2)u_2=\mu_2u_2^3+\beta u_2u_1^2 \quad & \R^N
\end{array}\right.
\]
satisfies
\[
\|u_i\|_{L^\infty(\R^N)}\leq C, \qquad u_i(x)\leq C e^{-\sqrt{c_0}(|x|-R_0)}, \ |x|\geq R_0, \qquad i=1,2.
\]
\end{lemma}
\begin{proof}
\underline{Uniform bounds}. Let $m=(\mu_2/\mu_1)^{1/4}$. Proceeding as in
\cite[Theorem 2.1]{MR2629888}, we see that there exists $\delta>0$ such that
\begin{equation}\label{eq:dww}
m(\mu_1 u_1^3+\beta u_1 u_2^2)+\mu_2 u_2^3+\beta u_1^2 u_2 \leq -\delta (m u_1+u_2)^3,
\end{equation}
for every $u_1,u_2>0$. Let
\[
M=\max_{x\in\R^N} (-V_1(x)-\omega_1,-V_2(x)-\omega_2).
\]
Notice that $M>0$ since
\[
\begin{split}
\int_{\R^N}[(V_1(x)+\omega_1)u_1^2+(V_2(x)+\omega_2)u_2^2]\,dx = \\
=\int_{\R^N}[-|\nabla u_1|^2-|\nabla u_2|^2+\mu_1u_1^4+\mu_2 u_2^4 +2\beta u_1^2 u_2^2]\,dx <0,
\end{split}
\]
and hence we can define
\[
z=\sqrt{\delta}(m u_1+u_2)-\sqrt{M}.
\]
By applying in turn Kato's inequality, the definition of $M$ and \eqref{eq:dww}, we have (here $\chi_U$ is the characteristic function of the set $U\subset\R^N$)
\[
\begin{split}
\Delta z^+ \geq \chi_{\{z\geq0\}} \sqrt{\delta} [ m u_1(V_1(x)+\omega_1)-m(\mu_1 u_1^3+\beta u_1 u_2^2) \\
+ u_2(V_2(x)+\omega_2) -(\mu_2 u_2^3+\beta u_1^2 u_2)] \\
\geq \chi_{\{z\geq0\}} \sqrt{\delta} (m u_1+u_2)[-M+\delta (m u_1+u_2)^2].
\end{split}
\]
We replace the definition of $z$ in the right hand side above to obtain
\[
\Delta z^+ \geq \chi_{\{z\geq0\}} (z+\sqrt{M})[-M+(z+\sqrt{M})^2] \geq (z^+)^3,
\]
which implies $z^+\equiv 0$ by the non-existence result \cite[Lemma 2]{Brezis84}, and hence the $L^\infty$-bounds.

\underline{Decay at infinity}. By the previous step and by \eqref{eq:V_assumptionT}, there exist $c_0,R_0>0$ such that
\[
V_i(x)+\omega_i-\beta u_j^2 \geq V_i(x) +\omega_i -|\beta| \|u_j\|_{L^\infty(\R^N)}^2 \geq c_0, \quad |x|\geq R_0,
\]
$i=1,2$, $j\neq i$. Then
\[
-\Delta u_i +c_0 u_i \leq \mu_i u_i^3 <0, \quad |x|\geq R_0,
\]
$i=1,2$. Let
\[
W_i(r)=\|u_i\|_{L^\infty(\R^N)}e^{-\sqrt{c_0}(r-R_0)}, \quad r=|x|\geq R_0,
\]
then we have $W_i(R_0)\geq \max_{\partial B_{R_0}}u_i$ and
\[
-\Delta W_i+c_0 W_i \geq 0, \quad r\geq R_0.
\]
By the maximum principle (which applies thanks to \eqref{eq:decay_u}), we deduce that $u_i(x) \leq W_i(|x|)$ for $|x|\geq R_0$.
\end{proof}

We define a map $\Phi:\mathcal{H}\times \R^4\to \mathcal{H}^*\times \R^3$ acting on $(u_1,u_2,\omega_1,\omega_2,\gamma,\alpha)\in \mathcal{H}\times \R^4$ as follows
\begin{equation}\label{eq:F_def}
\begin{split}
\text{for } i=1,2 \qquad & \Phi_i=\Delta u_i-(V_i(x)+\omega_i) u_i + \gamma u_i(\m_i u_i^2+\b u_j^2), \quad j\neq i \\
\text{for } i=3,4 \qquad & \Phi_i=\int_\Omega u_i^2\,dx - \rho_i, \\
& \Phi_5=\|(u_1,u_2)\|_{\mathcal{H}}^2 -\alpha.
\end{split}
\end{equation}

\begin{lemma}\label{lemma:defocusing_weak_nondegenerate}
If $(u_1,u_2,\omega_1,\omega_2,\gamma,\alpha)\in \mathcal{S}$, then the linear bounded operator
\[
L=\Phi_{(u_1,u_2,\omega_1,\omega_2,\gamma)}(u_1,u_2,\omega_1,\omega_2,\gamma,\alpha): \mathcal{H}\times\R^3 \to \mathcal{H}^*\times \R^3
\]
is invertible.
\end{lemma}
\begin{proof}
By Fredholm's Alternative, it will be enough to prove that $L$ is injective.
$L$ acts on $(v_1,v_2,o_1,o_2,g)$ as follows:
\[
L_i=\Delta v_i-(V_i(x)+\omega_i) v_i-o_i u_i + g u_i(\m_i u_i^2+\b u_j^2) +\gamma (3\m_i u_i^2 v_i+\b u_j^2 v_i +2\b u_1 u_2 v_j),
\]
for $i=1,2$, $L_i=2\int_\Omega u_i v_i \,dx$ for $i=3,4$, and
\[
L_5=2\int_\Omega (\nabla u_1\cdot\nabla v_1+V_1(x)u_1v_1+
\nabla u_2\cdot\nabla v_2+V_2(x)u_2v_2)\,dx.
\]
Suppose that $L(v_1,v_2,o_1,o_2,g)=0$. Testing the equation for $u_i$ by $v_i$, taking the sum for $i=1,2$ and using $L_3=L_4=L_5=0$ we find
\begin{equation}\label{eq:auxiliary_L_injective1}
\sum_{\substack{i=1 \\ j\neq i}}^2 \int_\Omega u_i v_i (\m_i u_i^2+\b u_j^2)\,dx =0.
\end{equation}
Testing the equation $L_i=0$ by $v_i$ for $i=1,2$, taking the sum and using the previous equality, we obtain
\begin{equation}\label{eq:auxiliary_L_injective2}
\sum_{\substack{i=1 \\ j\neq i}}^2 \int_\Omega \left\{ |\nabla v_i|^2 + (V_i(x)+\omega_i) v_i^2 -\gamma v_i^2(3\m_i u_i^2+\b u_j^2) \right\}\,dx -4\gamma \b \int_\Omega u_1u_2v_1v_2 \,dx=0.
\end{equation}
On the other hand, testing the equation for $u_i$ by $v_i^2/u_i$ leads to (the boundary term vanishes as in \eqref{eq:boundary_term_R_k})
\begin{equation}\label{eq:auxiliary_L_injective3}
\begin{split}
\int_\Omega\left\{ -(V_i(x)+\omega_i) +\gamma (\m_i u_i^2+\b u_j^2) \right\}v_i^2\,dx= \int_\Omega \nabla u_i\cdot\nabla\left( \frac{v_i^2}{u_i} \right)\,dx \\
=-\int_\Omega\left| \frac{v_i}{u_i}\nabla u_i-\nabla v_i\right|^2\,dx+\int_\Omega |\nabla v_i|^2\,dx \leq \int_\Omega |\nabla v_i|^2\,dx.
\end{split}
\end{equation}
Taking the sum for $i=1,2$ and exploiting \eqref{eq:auxiliary_L_injective2} we obtain
\[
\sum_{\substack{i=1 \\ j\neq i}}^2 \int_\Omega \m_i u_i^2 v_i^2 \,dx + 2\b \int_\Omega u_1 u_2 v_1 v_2 \,dx \geq 0.
\]
This, together with \eqref{eq:kappa_def}, implies $-\kappa \int_\Omega (u_1^2 v_1^2 + u_2^2 v_2^2)\,dx \geq 0$, and hence $v_1\equiv v_2 \equiv0$. In turn, the equations $L_1=L_2=0$ become $g(\m_i u_i^2+\b u_j^2)=o_i$, $i=1,2$, which provides $g=o_1=o_2=0$.
\end{proof}

Reasoning as in the previous proof, we also have the following related result regarding non degeneracy.

\begin{lemma}\label{eq:nondegeneracy_m_i<0}
Given $\omega_1,\omega_2\in\R$ and $\gamma>0$, the positive solution $(u_1,u_2)\in\mathcal{H}$ of
system \eqref{eq:system_main} is non degenerate as a critical point of the action functional
\begin{equation}\label{eq:action_functional_def}
\mathcal{A}_{\gamma,\omega_1,\omega_2}(u_1,u_2)=
\frac12\|(u_1,u_2)\|^2_{\Hcal}-\gamma F(u_1,u_2)+\frac{\omega_1}{2}
\mathcal{Q}_1(u_1)+\frac{\omega_2}{2}\mathcal{Q}_2(u_2).
\end{equation}
\end{lemma}
\begin{proof}
Take $(v_1,v_2)\in \mathcal{H}$ such that
\[
-\Delta v_i+(V_i(x)+\omega_i) v_i=\gamma(3\m_iu_i^2v_i+\b u_j^2v_i+2\b u_1u_2v_j) \qquad \text{ for } i,j=1,2, \,  j\neq i.
\]
By testing the equation of $v_i$ by $v_i$ itself, integrating by parts and summing up, we are lead to \eqref{eq:auxiliary_L_injective2}. Following the previous proof, we can then obtain once again that $-\kappa \int_\Omega (u_1^2 v_1^2 + u_2^2 v_2^2)\,dx \geq 0$, and hence $v_1\equiv v_2 \equiv0$, which proves the claim.
\end{proof}
\begin{lemma}\label{lem:defocusing_weak_gamma_increasing}
If $(u_1,u_2,\omega_1,\omega_2,\gamma,\alpha)\in\mathcal{S}$ then $\gamma'(\alpha)>0$ for every $\alpha>\lambda_{V_1}\rho_1+\lambda_{V_2}\rho_2$.
\end{lemma}
\begin{proof}
Fix $\alpha$ and consider the corresponding $(u_1,u_2,\omega_1,\omega_2,\gamma)$. Observe that, due to the assumptions on $\b,\m_1,\m_2$, the functional $\mathcal{A}_{\gamma,\omega_1,\omega_2}$ admits a global minimum in $\mathcal{H}$.
From the uniqueness result of Lemma \ref{lem:uniqueness_mu_i<0}, we deduce that actually
\[
\min_{\mathcal{H}} \mathcal{A}_{\gamma,\omega_1,\omega_2}
=\mathcal{A}_{\gamma,\omega_1,\omega_2}(u_1,u_2).
\]
By combining this with the non degeneracy result of Lemma \ref{eq:nondegeneracy_m_i<0}, we have that
\begin{equation}\label{eq:action_convex}
\mathcal{A}''_{\gamma,\omega_1,\omega_2}(u_1,u_2)[(\phi_1,\phi_2),(\phi_1,\phi_2)]>0 \qquad \forall \, (\phi_1,\phi_2)\neq (0,0).
\end{equation}
Thanks to Lemma \ref{lemma:defocusing_weak_nondegenerate}, we can locally differentiate the elements of $\mathcal{S}$ with respect to $\alpha$. Let
\[
\frac{d}{d\alpha} (u_1(\alpha),u_2(\alpha))
=:(v_1(\alpha),v_2(\alpha)).
\]
Then for $i,j=1,2$, $j\neq i$, we have
\begin{equation}\label{eq:identity1}
-\Delta v_i+(V_i(x)+\omega_i) v_i+\omega_i' u_i=\gamma(3\m_iu_i^2v_i+\b v_iu_j^2+2\b u_1u_2v_j)+\gamma' u_i (\m_iu_i^2+\b u_j^2)
\end{equation}
and identity \eqref{eq:identity_1/2} hold. By taking $(\phi_1,\phi_2)=(v_1,v_2)$ in \eqref{eq:action_convex}, and using \eqref{eq:identity1}, \eqref{eq:identity_1/2}, we deduce
\[
\begin{split}
\mathcal{A}''_{\gamma,\omega_1,\omega_2}(u_1,u_2)[(v_1,v_2),(v_1,v_2)]
    = \sum_{i=1}^2\int_\Omega (|\nabla v_i|^2+(V_i(x)+\omega_i) v_i^2-3\gamma \m_i u_i^2v_i^2)\, dx \\
 -\gamma \b \int_\Omega (v_1^2u_2^2+4u_1u_2v_1v_2+u_1^2v_2^2)\, dx \\
  =\gamma' \int_\Omega(\m_1 u_1^3v_1+\m_2u_2^3v_2+\b u_1u_2(v_1u_2+u_1v_2))\, dx =\frac{\gamma'}{2\gamma}>0,
\end{split}
\]
which yields $\gamma'>0$.
\end{proof}

\begin{remark}
In the assumptions of the previous lemma, proceeding very similarly to \cite[Lemma 5.6]{NorisTavaresVerzini2013}, it is also possible to prove that $ \omega_1'(\alpha)\rho_1+\omega_2'(\alpha)\rho_2 <0 $.
\end{remark}

\begin{proof}[End of the proof of Theorem \ref{thm:defocusing_weak_interaction}]
Combining Lemma \ref{lemma:defocusing_weak_nondegenerate} with Lemma \ref{lem:small_mass}, and
proceeding as in \cite[Proposition 5.4]{NorisTavaresVerzini2013} we obtain that $\mathcal{S}$ is a
smooth curve which can be parameterized by a unique map in $\alpha$. Theorem \ref{thm:intro1} and
\ref{thm:intro2} apply, providing the existence (and uniqueness) of the corresponding family of
standing waves, which are stable by Lemma \ref{lem:defocusing_weak_gamma_increasing}.  Finally,
by minimizing the energy
\[
\mathcal{E}_{\gamma}(u_1,u_2)=\frac12\|(u_1,u_2)\|^2_{\Hcal}-\gamma F(u_1,u_2)
\]
with $\mathcal{Q}(u_i)=\rho_i$, we obtain existence of elements
of $\mathcal{S}$ for every $\gamma>0$
\end{proof}

\small
\subsection*{Acknowledgments}
H. Tavares is supported by Funda\c c\~ao para a Ci\^encia e Tecnologia
through the program Investigador FCT. G. Verzini is partially supported  by the
PRIN-2012-74FYK7 Grant: ``Variational and perturbative aspects of nonlinear differential problems''.
The three of us are partially supported by the project ERC
Advanced Grant  2013 n. 339958: ``Complex Patterns for Strongly Interacting Dynamical Systems -
COMPAT''.


\end{document}